\documentclass{amsart}
\usepackage[margin=1.25in, marginparwidth=75pt]{geometry}
\usepackage[T1]{fontenc}
\usepackage{graphicx} 
\usepackage[utf8]{inputenc}
\usepackage{amsmath,amsthm,amsfonts,amscd}
\usepackage{tikz-cd}
\usepackage{dsfont}
\usepackage{mathrsfs}
\usepackage{upgreek}
\usepackage[square,numbers]{natbib}
\usepackage{pdflscape}
\usepackage{rotating}
\usepackage{eucal} 	 	
\usepackage{verbatim}      	
\usepackage{markdown}
\usepackage{booktabs}
\usepackage{url}
\usepackage{multirow}
\usepackage{makecell}
\usepackage{tabularx}
\usepackage{mathtools}
\usepackage{hyperref}
\hypersetup{colorlinks=false}
\usepackage{hhline}
\usepackage{tcolorbox}
\usepackage{nicematrix}
\usepackage{subcaption}
\usepackage{appendix}
\usepackage{svg}
\usepackage{enumitem,amssymb}
\newlist{todolist}{itemize}{2}
\setlist[todolist]{label=$\square$}

\newtheorem{theorem}{Theorem}[section]

\newtheorem{prop}{Proposition}[theorem]
\newtheorem*{prop*}{Proposition}
\newtheorem{lem}[prop]{Lemma}

\newtheorem*{conjecture*}{Conjecture}
\newtheorem*{guiding}{Guiding Principle}
\newtheorem{corollary}[prop]{Corollary}

\newtheorem*{question*}{Question}

\newtheorem*{theorem*}{Theorem}

\theoremstyle{remark}
\newtheorem*{remark*}{Remark}

\renewcommand{\H}[1]{\mathbb{H}^{#1}}
\newcommand{\slr}[1]{\mathfrak{sl}_{#1}(\mathbb{R})}
\newcommand{\son}[1]{\mathfrak{so}(#1,1)}
\newcommand{\pantsc}{\mathcal{P}}
\newcommand{\borro}{M}
\newcommand{\gad}{\mathfrak{g}_{Ad_\rho}}
\newcommand{\sslash}{\mathbin{/\mkern-6mu/}}

\title{Branched Bending in Finite-Volume Hyperbolic Manifolds}
\author{Casandra D. Monroe}
\address{Department of Mathematics, University of Michigan, Ann Arbor, MI}
\email{cdmonroe@umich.edu}

\begin{document}

\begin{abstract}
We define \emph{branched bending deformations} as deformations supported on a piecewise totally geodesic complex of $(n-1)$-dimensional faces meeting along $(n-2)$-dimensional branching loci. These are a generalization of bending deformations, as introduced by \citet{Johnson1987-ik}. We give a lower bound on the dimension of the (infinitesimal) deformation space supported on a branched bending complex, and in doing so generalize a result of \citet{Bart2006-im}. We give equations describing these deformations in the setting of deforming to higher hyperbolic geometry and real projective geometry. 
As a special example of branched bending, we construct infinitesimal deformations supported on the link complement of the Borromean Rings (also known as the link $6^3_2$), recovering a special case of a theorem due to \citet{menasco1992totally}. 
\end{abstract}

\maketitle
\vspace{-.5cm}
\section{Introduction}
It is a central question in the study of geometric topology to determine when the complete structure on a finite volume hyperbolic $n$-manifold deforms into a geometric structure of another type. 
Bending deformations, as introduced by \citet{Johnson1987-ik}, are explicit examples of such deformations. The construction makes use of an embedded totally geodesic hypersurface $\Sigma$ and, in return, yields a family of deformations into a larger semisimple Lie group $G$.

While these examples explain some of the flexibility of certain hyperbolic manifolds, there are some that deform even in the absence of an embedded totally geodesic hypersurface. Cooper, Long, and Thistlethwaite (\cite{cooper2006computing}, \cite{Cooper2007-cs}) analyzed over 4,500 closed hyperbolic $3$-manifolds and showed that 61 admit (infinitesimal) deformations into $SL_{4}(\mathbb{R})$, despite not having any embedded (or even immersed!) totally geodesic surfaces. There are similar examples of such manifolds supporting deformations into $SO^+(4,1)$. 

In light of these examples, one can then ask: what topological features of a hyperbolic manifold contribute to its flexibility? \citet{Bart2006-im} study a generalization of bending: they find an example of a cusped hyperbolic $3$-manifold that contains no embedded totally geodesic surfaces, and describe a piecewise totally geodesic 2-complex that supports an infinitesimal deformation into $SO^+(4,1)$. Furthermore, they establish a lower bound on the dimension of the space of \emph{infinitesimal} deformations in terms of the number of 2-dimensional pieces $c_2$ and the number of 1-dimensional pieces $c_1$ of a totally geodesic 2-complex. Inspired by their examples, we will define \emph{branched bending deformations} as deformations supported on a piecewise totally geodesic complex of $(n-1)$-dimensional faces meeting along $(n-2)$-dimensional branching loci. 

In Theorem \ref{thm:bbson1}, we give an explicit geometric interpretation of the bound given by Bart and Scannell, as well as a criterion for when such deformations are possible from $\son{n}$ into $\son{n+1}$, when $n\geq3$. 
\begin{theorem} \label{thm:bbson1}
Suppose $M = \Gamma \backslash \mathbb{H}^n$ is a complete hyperbolic $n$-manifold containing a branched totally geodesic hypersurface $\Delta$ with branching locus $\mathscr{B}$. Let $c_{(n-1)}$ be the number of $(n-1)$-dimensional complementary regions in $\Delta\backslash \mathscr{B}$, and let $c_{(n-2)}$ be the number of components of $\mathscr{B}$. Then, with mild conditions on $\Tilde{M}\backslash \Tilde{\Delta}$, we obtain:
\[dim_{\mathbb{R}} H^{1}(\Gamma, \mathbb{R}^{n,1}) \geq c_{(n-1)} - 2c_{(n-2)}\] 
\end{theorem} 
Furthermore, we produce a system of equations that are a function of the local geometric data and whose output is an infinitesimal deformation of some representation $\rho$ from $\son{n}$ into $\son{n+1}$. The conditions on $\Tilde{M}\backslash \Tilde{\Delta}$ ensure that the components in the complement of the hypersurface $\Delta \subset M$ have ``enough topology'' in some sense. These conditions can be thought of as analogous to those described in Lemma 5.4 of \cite{Johnson1987-ik}. The precise statement of Theorem \ref{thm:bbson1} can be found in Theorem \ref{thm:bbson1fin}. 

We also consider higher rank deformations. In Theorem \ref{thm:bbpgln}, we develop similar bounds and techniques for infinitesimal branched deformations from $\son{n}$ into $\slr{n+1}$. 
\begin{theorem} \label{thm:bbpgln}
Let $M$, $\Delta$, and $\mathscr{B}$, be as in Theorem \ref{thm:bbson1}. 
Then 
$$dim_{\mathbb{R}} H^{1}(\Gamma, \nu_{n+1}) \geq c_{(n-1)} - 3c_{(n-2)}$$
\end{theorem} 
Similarly, we produce a system of equations that are a function of the local geometric data and whose output is an infinitesimal deformation of some representation $\rho$ from $\son{n}$ into $\slr{n+1}$. The precise statement can be found in Theorem \ref{thm:bbpglnfin}.

In Theorem \ref{thm:bbson1} and Theorem \ref{thm:bbpgln}, the choice of coefficients is motivated by a splitting of the Lie algebras of $\son{n+1}$ and $\slr{n+1}$ respectively. For further discussion of this splitting, we refer the reader to Section \ref{subsec:splitting}.  

%%%%%%%%%%%%%%%%%%%%%%%%%%%%%%%%%%%%%%%%%%%%%%%%%%%%%%%%
\subsection*{Overview}
The question of how much of the deformation space of a hyperbolic manifold is described by bending goes back to its introduction. The bending construction defined by Johnson and Millson \cite{Johnson1987-ik} makes use of an embedded totally geodesic hypersurface $\Sigma$ and, in return, yields a family of deformations into a larger semisimple Lie group $G$ parametrized by $Z_{G}(SO^+(n-1,1))$. When $G=SO^+(n+1,1)$ or $G=SL_{n+1}(\mathbb{R})$, bending yields a 1-parameter family of deformations of the original complete hyperbolic structure of the manifold. In these settings, we have the following theorem: 

\begin{theorem} [\cite{Johnson1987-ik}] \label{thm:jam}
    Suppose a manifold $M = \Gamma \backslash \mathbb{H}^n$ contains $r$ disjoint embedded two-sided connected totally geodesic hypersurfaces $M_{1}, M_{2}, \dots, M_{r}$. Then we have that the dimension of deformations into $G$ is at least $r$. 
\end{theorem}

This theorem motivated the following conjecture: 

\begin{conjecture*}[Kourouniotis, \cite{Mathematisches_Forschungsinstitut_Oberwolfach1985-bo}] \label{thm:allbend}
    %The space of (spatial) deformations of a manifold $\mathbb{H}^n/\Gamma$ is exactly the bending deformations.
    The only nontrivial deformations of the complete hyperbolic structure of a compact hyperbolic $n$-manifold for $n>2$ are those that arise from bending.  
\end{conjecture*}

We now understand this conjecture to be false; there are several hyperbolic manifolds for which we understand the deformation space cannot be entirely described by bending (if at all). For example, several of the flexible manifolds established by \citet{Cooper2007-cs} possess even integrable deformations of their complete hyperbolic structure in the absence of an embedded, totally geodesic submanifold. These examples inspire a sort of realization problem: what topological properties of the manifold give rise to these deformations of the geometry? 

We look to Theorem \ref{thm:jam} as inspiration. Which hypotheses, if any, can be weakened such that we get a more robust description of what deformations are possible? In this paper, we investigate removing the condition that the totally geodesic hypersurfaces need be disjoint. Instead, we allow them to intersect, and more generally, to form branched hypersurfaces. 

The direct generalization of bending along \textit{intersecting} surfaces appears more complicated---the usual splitting of the fundamental group called for in the construction is not as simple to arrange. In exchange for a more general bending locus, we expect more restrictions on the potential deformations. With this in mind, we look to defining \emph{infinitesimal} deformations supported along our branched bending loci. This approach was explored in the original paper of \citet{Johnson1987-ik}, where they showed that the composition of infinitesimal bending deformations supported along intersection hypersurfaces is (typically) not integrable. (For a special example of when such a deformation is integrable, see Corollary \ref{cor:proj}).

Inspired by the aforementioned conjecture of Kourouniotis, Apanasov introduced a different notion of deformation in a hyperbolic $n$-manifold called a ``stamping'' deformation. 
A stamping deformation is a deformation supported on an arrangement of intersecting totally geodesic hypersurfaces, and was initially described as distinct from bending \cite{Apanasov1990-lc}. However, Bart and Scannell showed that a stamping deformation is actually the composition of (three) bending deformations \cite{Bart2007-kc}. In related work to this notion of stamping and bending supported along arrangements of intersecting hypersurfaces, and in analogous form to Theorem \ref{thm:jam}, Bart and Scannell prove: 

\begin{theorem}[Theorem 3.1 in \cite{Bart2006-im}] \label{thm:bas}
Suppose $M=\Gamma \backslash \mathbb{H}^3$ is a complete hyperbolic $3$-manifold containing a branched totally geodesic hypersurface $S$ with branch locus $B$. Let $c_2$ be the number of two-dimensional complementary
regions in $S\backslash B$ with non-elementary Fuchsian fundamental group, and let $c_1$ be the number of components of $B$. Then the space of infinitesimal bending deformations supported on $S$ (a fortiori $H^1(\Gamma, \mathbb{R}^{3,1})$)
 has dimension at least $c_2 - 2c_1$.
\end{theorem} 
The approach taken in \cite{Bart2006-im} is very algebraic; for example, the proof of Theorem \ref{thm:bas} makes use of spectral sequences, a direct analogue of the long exact sequences of Johnson and Millson. 
%This result is also in the particular setting of hyperbolic $3$-manifolds. 

There have been other strides made in the direction of this notion of generalized bending, such as in \cite{apanasov1992deformations}, \cite{Kapovich1996-sy}, \cite{tan1993deformations} and \cite{Maubon2000-ic}. The results in \cite{Bart2006-im} and \cite{tan1993deformations} are specific to $3$-manifolds. The results in \cite{apanasov1992deformations}, \cite{Maubon2000-ic} use branched bending complexes that are slightly more constrained than those in Theorem \ref{thm:bas}. Theorems \ref{thm:bbson1} and \ref{thm:bbpgln} describe branched bending in dimension $n\geq2$, with an even more generalized notion of branched bending complex.

%%%%%%%%%%%%%%%%%%%%%%%%%%%%%%%%%%%%%%%%%%%%%%%%%%%

\subsection*{Application to the Menasco-Reid Conjecture}
For finite-volume manifolds, one can ask about deformations that do not affect the cusps, as in \cite{Bart2006-im, porti2013local, Scannell2002-ve}. In this case, bending along a \emph{closed} embedded totally geodesic hypersurface remains the most well-understood way to obtain explicit deformations of finite-volume hyperbolic $n$-manifolds into other geometries, when $n\geq3$. Complements of hyperbolic knots and links are a natural class of finite-volume manifolds to consider, and in that setting we have the following question:

\begin{conjecture*}[\citet{menasco1992totally}] \label{thm:mrc}
    There does not exist a hyperbolic knot complement in $S^{3}$ that contains a closed embedded totally geodesic surface.
\end{conjecture*}

While it is now known that there exist counterexamples to this conjecture \cite{Deblois2025-pv}, the question of classifying what hyperbolic knots and links fail to admit closed, embedded totally geodesic surfaces remains interesting.
For example, in the same work where they propose this conjecture, Menasco and Reid show that it is satisfied for alternating knots, tunnel number one knots, knots and links of braid index $3$ (also proved in \cite{lozano1985incompressible}), and knots with 2-generator fundamental group. 
%The conjecture has also been verified for almost alternating knots \cite{adams1992almost}, toroidally alternating knots \cite{Adams1994-ck}, Montesinos knots \cite{Oertel1984-fn}, 3-bridge knots and double torus knots \cite{Ichihara2000-lk}, and knots of braid index $4$ \cite{matsuda2002complements}. 
The conjecture has been computationally verified for knots of up to $12$ crossings \cite{basilio_et_al}. 

We note that the conjecture was known not to hold in general for links. Menasco and Reid present an $8$-component link whose complement admits a closed embedded totally geodesic surface of genus $2$ \cite{menasco1992totally}. In fact, for any integer $g\geq2$, there exists a 2-component hyperbolic link that contains a closed embedded totally geodesic surface of genus $g$ \cite{leininger2006small}. However, there are results guaranteeing some families of hyperbolic links satisfy the Menasco-Reid conjecture, such as those that arise as the closure of $3$-braids \cite{menasco1992totally}. 

Motivated by the remaining question of what knots and links fulfill the Menasco-Reid Conjecture, we describe here an approach that makes use of deformation theory. To show that a knot or link complement does not contain a closed embedded totally geodesic surface is to show the manifold has vanishing \emph{cuspidal cohomology}. In other words, that the infinitesimal deformations of the knot exterior relative to the cusps are trivial, as studied in \cite{Bart2006-im}. A closed embedded totally geodesic hypersurface in the exterior would support a bending deformation that does not affect the cusps, and this would contribute to the dimension of the cuspidal cohomology. As an application of this argument, in \cite{Bart2006-im, kapovich1994deformations}, it is shown that cuspidal cohomology vanishes for $2$-bridge knots, and therefore that for this class of knots, there are no closed embedded totally geodesic surfaces. Here, we also compute cuspidal cohomology as an approach to construct new examples of knot and link complements with interesting properties, with a particular eye towards the Menasco-Reid conjecture. 

In the proof of both Theorems \ref{thm:bbson1} and \ref{thm:bbpgln}, we construct a system of equations which are a function of the local geometric data, and which output an infinitesimal deformation of the initial representation. We make use of these equations in analysis of the Borromean rings complement $S^{3} \backslash 6^{3}_{2}$. This link complement was previously known to not contain any closed, embedded totally geodesic surfaces \cite{menasco1992totally}; we recover this as in Corollary \ref{cor:bor}. This follows directly from Theorem \ref{thm:H1H4}, where we show the infinitesimal deformation space relative to the cusps in the $\son{4}$ setting is trivial. In contrast with this fact, Theorem \ref{thm:HPG} shows that both the dimension of the infinitesimal deformation space and the dimension of the infinitesimal deformation space relative to the cusps are larger in the $\slr{4}$ setting than in the $\son{4}$ setting. 

\begin{theorem} \label{thm:H1H4} 
    For $\Gamma = \pi_{1}(S^{3} \backslash 6^{3}_{2})$, $dim_{\mathbb{R}} H^{1}(\Gamma, \mathbb{R}^{3,1}) = 3$ and $dim_{\mathbb{R}} PH^{1}(\Gamma, \mathbb{R}^{3,1}) = 0$. 
    This space of infinitesimal deformations is spanned by bending deformations supported along six totally geodesic thrice-punctured spheres that intersect to form a non-compact branched bending complex. 
\end{theorem}

\begin{corollary}\label{cor:bor} The Borromean rings complement does not contain any closed, embedded totally geodesic surfaces. 
\end{corollary}

\begin{theorem} \label{thm:HPG}
   For $\Gamma = \pi_{1}(S^{3} \backslash 6^{3}_{2})$:
   \[  dim_{\mathbb{R}} H^{1}(\Gamma, \nu_{4}) = 6, \qquad  dim_{\mathbb{R}} PH^{1}(\Gamma, \nu_{4}) = 3 \]
    %\label{thm:PH1PG}
Each of these spaces of infinitesimal deformations is parametrized by bending deformations supported along the complex described in Theorem \ref{thm:H1H4}. 
\end{theorem}

By combining the above with established results in convex projective geometry, we are able to show:
\begin{corollary} \label{cor:proj}
The deformations described by $PH^{1}(\Gamma, \nu_{4})$  are all integrable; furthermore, they are all strictly convex projective deformations. In particular, the holonomy representation $\rho \in \chi(\Gamma, SL_4(\mathbb{R}))$ is a smooth point in the character variety for the Borromean rings complement. 
\end{corollary}

\subsection*{Organization}

In Section \ref{prelim}, we will give a review of bending deformations, especially in the settings of $SO(n,1)$ and $SL_n(\mathbb{R})$, as well as a discussion of group cocycles and their relationship to infinitesimal deformations. In Section \ref{sec:book}, we establish a formalism adapted from \cite{Danciger2016-qq} in order to describe infinitesimal deformations. In Section \ref{sec:low}, we describe a general form of lower bound of the dimension of a deformation space for a hyperbolic manifold with a given branched bending complex. 
In Section \ref{sec:prod}, we prove Theorems \ref{thm:bbson1} and \ref{thm:bbpgln}; in doing so, we define and explore the construction of branched bending deformations, motivated by the works of \cite{Apanasov1990-lc}, \cite{Bart2006-im}, and \cite{Bart2007-kc}. In Section \ref{sec:defbor}, we go over a special example of the ideas developed in Theorems \ref{thm:bbson1} and \ref{thm:bbpgln} by looking to cusped hyperbolic $3$-manifolds, and prove Theorems \ref{thm:H1H4} and \ref{thm:HPG}.

\subsection*{Acknowledgments}

The author wishes to thank Fernando Camacho-Cadena, James Farre, Daniel Allcock, Marit Bobb, Sara Maloni, and Anna Wienhard for helpful conversations on earlier versions of this project. Special thanks to Yair Minsky, for enlightening conversation on 3-manifolds, and Joan Porti, for pointing to the Borromean rings as a potential example of interesting phenomena. The author is deeply grateful to Dick Canary for feedback on earlier drafts of this paper. Finally, the author wishes to thank their advisor Jeff Danciger for many emails, edits, and invaluable conversations.

\section{Background} \label{prelim}

\subsection{Geometric Structures and Bending Deformations}
Let $\Gamma$ be a torsion-free lattice in Isom$(\mathbb{H}^{n})$. Then $M=\mathbb{H}^{n}/\Gamma$ is a finite volume hyperbolic $n$-manifold. We identify $\Gamma$ with the image of $\rho:\pi_1(M) \to$ Isom$(\mathbb{H}^{n})$, where $\rho$ is a discrete and faithful representation. Let $\Sigma$ be an embedded, totally geodesic hypersurface in $M$, and let $S=\pi_1(\Sigma)$. 

Consider a semisimple Lie group $G$ such that Isom$(\mathbb{H}^{n}) \subseteq G$. We consider $C_{G}(S)$, which, for the choice of $G$ relevant to our applications, will be 1-dimensional (but is not so necessarily). When $C_{G}(S)$ is a 1-parameter family, we denote it $c(t)$, and also specify that $c(0)=I$.

When $\Sigma$ splits $M$ into $M_1$ and $M_2$, and thus $\pi_{1}(M)=\pi_{1}(M_{1}) \ast_{S} \pi_{1}(M_{2})$, then \emph{bending} the representation $\rho$ along $S$ is defined as: 
\[ 
\rho_t(\gamma) = 
\begin{cases} 
      \rho_{0}(\gamma), &  \gamma \in \pi_1(M_1) \\
       c(t)\rho_{0}(\gamma)c(t)^{-1}, & \gamma \in \pi_1(M_2)
   \end{cases}
\]

When $\Sigma$ is nonseparating, then  $\pi_{1}(M)=\pi_{1}(M_{\Sigma}) \ast_{\alpha}$, where $M_{\Sigma} = M\backslash\Sigma$ and $\alpha$ is the stable letter of the HNN extension induced by $\Sigma$. Then \emph{bending} the representation $\rho$ along $S$ is defined as: 
\[ 
\rho_t(\gamma) = 
\begin{cases} 
      \rho_{0}(\gamma), & \gamma \in \pi_1(M_{\Sigma}) \\
       c(t)\rho_{0}(\alpha), & \gamma = \alpha
   \end{cases}
\]

\subsubsection{$G=SO(n+1,1)$}
When $G=SO(n+1,1)$, the bending deforms the initial $\H{n}$-structure on $M$ to a $\H{n+1}$-structure. A  totally geodesic hypersurface $\Sigma$ will have a 1-parameter centralizer, which can be represented, up to conjugation, as follows:
\[
C_G(S) = \left\{c(t) = 
\left(
    \begin{NiceArray}{ccccc}
     \cos(t) & -\sin(t)  &    & &  \\
        \sin(t) & \cos(t) &  &  &      \\ 
        &&1 \Block[borders={top,left, tikz=dashed}]{3-3}{} && \\
          & &     & \ddots & \\
         &  &       &  & 1
    \end{NiceArray}
\right) 
\ 
\middle\vert
\
t \in \mathbb{R}
\right\}
\]
where the bottom right $n\times n$ block represents the copy of $\H{n-1}\subset \H{n}$ containing $\Tilde{\Sigma}$. 

\subsubsection{$G=SL_{n+1}(\mathbb{R})$}
When $G=SL_{n+1}(\mathbb{R})$, the bending deforms the initial $\H{n}$-structure on $M$ to a $\mathbb{RP}^{n}$ structure. A totally geodesic hypersurface $\Sigma$ will have a 1-parameter centralizer, which can be represented, up to conjugation, as follows:
\[ 
C_G(S) = \left\{c(t) =
\left(
    \begin{NiceArray}{cccc}
     e^{-nt} &   &    &   \\
             & e^{t}\Block[borders={top,left, tikz=dashed}]{3-3}{} &  &      \\
           &     & \ddots & \\
           &       &  & e^{t}
    \end{NiceArray}
 \right)
 \ 
\middle\vert
\
t \in \mathbb{R}
 \right\}
\] 
where the bottom right $n \times n$ block represents the copy of $\H{n-1}\subset \mathbb{RP}^{n}$ containing $\Tilde{\Sigma}$.

The original proof of Johnson and Millson verifies that these bending deformations are nontrivial by showing that the infinitesimal deformation tangent to this bending is identified with a nontrivial cocycle in group cohomology with appropriate coefficients. We take this opportunity to define this group cohomology in the next subsection.

\subsection{Group Cohomology and Infinitesimal Deformations}

To study discrete and faithful representations $\rho: \Gamma \to G$, we study the representation variety Hom$(\Gamma, G)$. We denote the character variety $\chi(\Gamma,G)=\text{Hom}(\Gamma,G)\sslash G$; that is, it is the representation variety considered up to conjugation by $G$. (As taking a quotient by $G$ may yield something non-Hausdorff, we take $\sslash$ to be the GIT-quotient, and refer the reader to \cite{Lubotzky1985-vm}, \cite{Luna1975-mq} for further explanation.) 

For a representation $\rho \in \chi(\Gamma,G)$, deformations of $\rho$ correspond to paths in $\text{Hom}(\Gamma,G)$ through $\rho$. To find such paths is nontrivial, and strides have been made doing this experimentally, such as in \cite{cooper2006computing}. However, to find infinitesimal deformations, we look to the Zariski tangent space $T_{\rho}\chi(\Gamma, G)$. The dimension of this tangent space is an upper bound on the number of distinct integrable deformations of $\rho$, and when $\rho$ is a smooth point in the representation variety, this bound becomes an equality. 

Thankfully, we can identify the Zariski tangent space $T_{\rho}\chi(\Gamma, G)$ with $H^1(\Gamma, \gad)$. To see this, let the smooth path $\rho_t$: $[0,1] \to \text{Hom}(\Gamma,G)$ be a deformation of $\rho=\rho_0$. Then the tangent vector to the deformation at $t=0$ can be described by assigning an element of the Lie algebra $\gad$ in the following manner:
\[ c(\gamma) = \dot{\rho}(\gamma) \rho(\gamma)^{-1} \]
That is, we take the derivative of $\rho(\gamma)$, which gives us an element of the tangent space $T_{\rho(\gamma)}G$. To land in the Lie algebra of G, $\mathfrak{g}$, we need to take values in $T_{e}G$, and so we correct for this with $\rho(\gamma)^{-1}$. Then we obtain a map $c: \Gamma \to \gad$.
Since each $\rho_t$ is a homomorphism, differentiation yields the following condition that our map $c$ satisfies: 
\[ c(\gamma_1 \gamma_2 ) = c(\gamma_1) + \mathrm{Ad}(\gamma_1)c(\gamma_2)\]
for all $\gamma_1, \gamma_2 \in \Gamma$. This condition is the cocycle condition, and so we see that each such $c$ as defined above is a map such that $c \in Z^1(\Gamma, \gad)$. Thus, we can identify $T_\rho$Hom$(\Gamma, G)$ with $Z^1(\Gamma, \gad)$.

The coboundaries in this construction are the tangent vectors to deformations arising from conjugation, giving rise to the coboundary condition: 
\[ c(\gamma) \in \text{im}(I-\text{Ad}_\gamma), \ \forall \gamma \in \Gamma.\] 

Together, we have 
\[ H^1(\Gamma, \gad) = Z^1(\Gamma, \gad) / B^1(\Gamma, \gad)\]
and in our setting, can identify this with $T_{\rho}\chi(\Gamma, G)$. (We note that this identification is possible in our setting because $\rho(\Gamma)$ has trivial centralizer in the group $G$ used throughout. For further discussion of these ideas, see \cite{Huebschmann2001-mn} and \cite{Weil1964-nb}.)

Throughout, we will be interested in computing the $\mathbb{R}$-dimension of this cohomology with appropriate coefficients as a way of understanding the infinitesimal deformation theory of the hyperbolic structure associated with $\rho$. 

In this paper, for $G=SO^{+}(n+1,1)$ or $SL_{n+1}(\mathbb{R})$ and a discrete and faithful representation $\rho: \Gamma \to G$, we construct an element $\xi \in Z^1(\Gamma, V)$ where $V=\mathbb{R}^{n,1}$ or $\nu_{n+1}$. These choices of $V$ arise as $\text{Ad}_{SO(n,1)}$-invariant subspaces of $\son{n+1}$ and $\slr{n}$ respectively. That is, there is an $Ad$-invariant splitting of $\mathfrak{so}(n+1,1)$,  so that
 \[ \mathfrak{so}(n+1,1) = \mathfrak{so}(n,1) \oplus \mathbb{R}^{n,1} \]
 while for $\slr{n+1}$, we have 
\[ \slr{n+1} = \mathfrak{so}(n,1) \oplus \nu_{n+1}. \]

\subsubsection{Parabolic and Cuspidal Cohomology} \label{subsec:splitting}

In Section \ref{sec:defbor}, we study special instances of group cohomology; we offer a brief definition of these terms here and direct the reader to Appendix \ref{app:cusp} for further discussion. 

We define \emph{parabolic cohomology} as follows. Let $\Gamma$ be a nonuniform lattice in $SO(n,1) \subset G$ such that $\mathbb{H}^{n}/\Gamma$ is a finite volume, cusped hyperbolic $n$-manifold $M$. If $\rho:\pi_1(M)\to G$, and im$(\rho) = \Gamma$ , then:
\[ PH^{1}(\Gamma , \mathfrak{g}_{Ad_{\rho}}) = PZ^{1}(\Gamma , \mathfrak{g}_{Ad_{\rho}})/B^{1}(\Gamma , \mathfrak{g}_{Ad_{\rho}}) \]
where 
\[PZ^{1}(\Gamma , \mathfrak{g}_{Ad_{\rho}}) = \{ c \in Z^{1}(\Gamma , \mathfrak{g}_{Ad_{\rho}})\  \vert\  c(\gamma) \in \text{im}(I-\text{Ad}_\gamma), \  \forall \gamma \ \text{parabolic} \} \]

On the other hand, we define the \emph{cuspidal cohomology} of $M$, denoted throughout here as $H^{1}_{\partial}(\Gamma , \mathfrak{g}_{Ad_{\rho}})$ as:
\[ \ker \Big( \text{res:}\  H^{1}((\Gamma\backslash\H{3}), \mathfrak{g}_{Ad_{\rho}}) \to H^{1}(\partial(\Gamma\backslash\H{3}), \mathfrak{g}_{Ad_{\rho}}) \Big)  \]
where $\mathfrak{g}$ is the Lie algebra of $G$.

\section{Bookkeeping for Cohomology Classes} \label{sec:book}
We adapt the formalism introduced in Section 4.2 of \cite{Danciger2016-qq} to our setting. The main idea of this section of their paper is to describe infinitesimal deformations of the hyperbolic structure of convex cocompact surfaces; to do so, they decompose these manifolds along a \emph{geodesic cellulation} into tiles. By looking at the preimage of these tiles in the universal cover, they characterize an infinitesimal deformation as an equivariant assignment of \emph{relative motions} of neighboring tiles. In our setting, we will be making precise how to describe an infinitesimal deformation supported on a \textit{branched bending complex} in a finite volume hyperbolic $n$-manifold $M=\Gamma\backslash\H{n}$.

We consider a \emph{branched geodesic decomposition} $\Delta$ of $M$ to be a decomposition along a branched bending complex; that is, a decomposition of $M$ into top-dimensional \emph{regions} $\mathscr{C}$, along codimension-$1$ pieces called \emph{walls} $\mathscr{W}$, meeting at pieces of codimension-$2$ called \emph{bindings} $\mathscr{B}$. As mentioned in the descriptions of Theorems \ref{thm:bbson1}, \ref{thm:bbpgln}, in our setting, each element in $\mathscr{W}$ is totally geodesic in $M$, with fundamental group with Zariski dense image in $SO^+(n-1,1)$. We pass to the universal cover $\Tilde{M}$ of $M$, and let $\Tilde{\Delta}$ be the preimage of $\Delta$; we use a similar convention in the denotation of $\Tilde{\mathscr{C}}$, $\Tilde{\mathscr{W}}$, and $\Tilde{\mathscr{B}}$. The elements of $\Tilde{\mathscr{C}}$ will be called \emph{cells}. As each element of $\Tilde{\mathscr{W}}$ is of codimension-$1$, we are able to assign these elements a transverse orientation. Borrowing from the language of \cite{Danciger2016-qq}, we let each element of $\Tilde{\mathscr{W}}$ along with a transverse orientation be denoted $\pm\Tilde{\mathscr{W}}$, and say for each $w \in \pm\Tilde{\mathscr{W}}$, $-w$ is the same wall with opposite transverse orientation. 

We make use of two maps: a map $\varphi: \Tilde{\mathscr{C}} \rightarrow V$ that describes the global motion of the cells, and a map $\psi: \pm\Tilde{\mathscr{W}} \rightarrow V$ that describes the \emph{relative} motion of the cells: $\psi(w)=\varphi(\delta') - \varphi(\delta)$ for $\delta\in \mathscr{C}$ on the negative side of $w$ and $\delta'\in \mathscr{C}$ on the positive side of $w$. 
%Together, these maps will help us construct a cocycle $\xi$ via integration.

\begin{figure}
    \centering
    \includegraphics[width=0.5\linewidth]{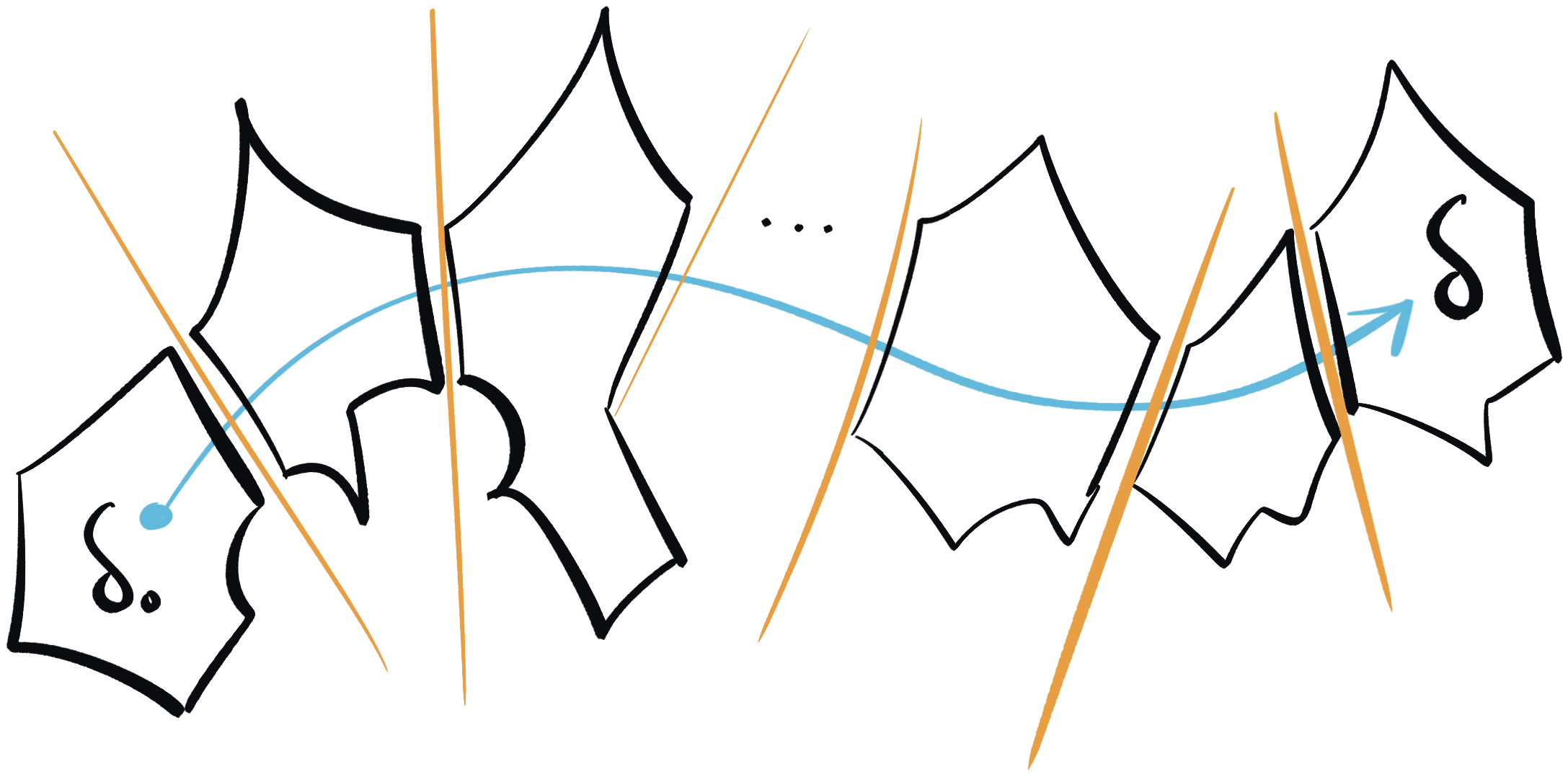}
    \caption{A path from $\delta_0$ to $\delta$, as in the definition of the $\varphi$ map }
    \label{fig:placeholder}
\end{figure}

We aim to ``integrate'' the $\psi$ map to construct an appropriate corresponding $\varphi$ map. With this in mind, we require the following \emph{consistency conditions} on $\psi$: 
\begin{enumerate}
    \item $\psi$ is $\rho$-equivariant, here meaning: 
\[ \psi(\gamma\cdot w) =  \rho(\gamma)\cdot \psi(w)\]
for all $\gamma \in \Gamma$ and for any $w\in \pm\Tilde{\mathscr{W}}$;

    \item $\psi(-w)=-\psi(w)$ for all $w \in \pm\Tilde{\mathscr{W}}$;

    \item the total motion around any binding is zero: that if there exists a path via a sequence of transversely oriented $w_1, \dots, w_k \in \pm\Tilde{\mathscr{W}}$ that forms a loop around some binding $b \in \mathscr{B}$, then  $\sum^k_i \psi(w_i)=0$.
\end{enumerate}
If $\psi$ satisfies all of the above conditions, then define:
\[ \varphi(\delta) := \zeta_0 + \sum \psi(w) \]
where $\delta \in \Tilde{\mathscr{C}}$ and $\varphi$ sums over all of the transversely oriented edges $w$ along some path starting in $\delta_0$ and ending in $\delta$. This construction makes use of initial data: a cell $\delta_0 \in \mathscr{C}$ and an initial motion $\zeta_0 \in V$. However, we will show later that our construction is ultimately independent of this choice.

\begin{lem}
Let $\psi$ be as defined above. Then the following formula well-defines a map $\varphi: \mathscr{C} \to V$ such that 
\[ \varphi(\delta) := \zeta_0 + \sum \psi(w) \]
where $\delta \in \Tilde{\mathscr{C}}$ and $\varphi$ sums all of the transversely oriented edges $w$ along some path starting in $\delta_0$ and ending in $\delta$ is independent of the choice of path.
\end{lem}
The heart of this proof comes down to the fact that any two paths ``differ by a loop'' in some sense. The result then follows from consistency condition (3). 
% \begin{proof}
%     Consider two paths $p$ and $p'$ both from $\delta_0$ to $\delta$. We want to show that value of $\varphi(\delta)$ is independent of this choice. We do this in an almost ``inductive" manner. That is, we can construct a (finite) family of paths $\{p_k\}$, $k\in \mathbb{N}$ so that $p_0 = p$, $p_N = p'$ and each $p_{i}, p_{i+1}$ differ by crossing over at most one binding. 

%     Then we can now argue that paths that differ by crossing over at most one binding yield the same value under $\varphi(\delta)$. The $p_i\  \cup -p_{i+1}$ is a loop around some binding $b$, and thus has total sum zero. And so, $\varphi_{p_{i}}(\delta) = \varphi_{p_{i+1}}(\delta)$.

%     Thus, there is a homotopy of $p$ to $p'$ consisting of paths that cross one binding at a time, where the value along all of these homotopic paths is preserved. 
% \end{proof}

We can now define our cocycle. Choose a cell $\delta_0 \in \mathscr{C}$ and an initial motion $\zeta_0 \in V$. Then
\[ \xi(\gamma) := \varphi( \gamma\cdot \delta_0) -  (\rho(\gamma)\cdot\varphi(\delta_0))\]
is a $\rho$-cocycle.

% To see $\xi$ defines a cocycle, we show that it satisfies the \emph{cocycle condition}. We want to show 
% \[ \xi(\alpha\beta ) = \xi(\alpha) +  \rho(\gamma)\cdot \xi(\beta)\]
% The LHS of this equality becomes 
% \begin{align*}
%     \xi(\alpha\beta ) &=  \varphi( \alpha\beta\cdot \delta_0) -  \rho(\alpha\beta)\varphi(\delta_0) 
% \end{align*}  
% while the RHS becomes
% \begin{align*}
%     \xi(\alpha) +  \rho(\alpha)\cdot \xi(\beta) &= \varphi( \alpha\cdot \delta_0) -  \rho(\alpha)\varphi(\delta_0) +  \rho(\alpha)\Big(\varphi(\beta\cdot \delta_0) -  \rho(\beta)\varphi(\delta_0)\Big) \\
%                 &= \varphi( \alpha\cdot \delta_0) -  \rho(\alpha)\varphi(\delta_0) +   \rho(\alpha)\varphi( \rho(\beta)\cdot \delta_0) -  \rho(\alpha\beta)\varphi(\delta_0) \\
%                 &= \varphi( \alpha\cdot \delta_0) +  \rho(\alpha)\Big(\varphi( \beta\cdot \delta_0) - \varphi(\delta_0)\Big) -  \rho(\alpha\beta)\varphi(\delta_0)
% \end{align*}
% and so it suffices to show that 
% \begin{equation}    \label{eq:cocycle}
%  \varphi( \alpha\beta\cdot \delta_0) =  \varphi( \rho(\alpha)\cdot \delta_0) +  \rho(\alpha)\Big(\varphi( \rho(\beta)\cdot \delta_0) - \varphi(\delta_0)\Big) 
%  \end{equation}
% However, one sees that the right of this equation truly is a decomposition of the sum on the left; the RHS can be seen as the sum of the $\psi$ values along the concatenation of paths $\delta_0$ to $\rho(\alpha) \delta_0$ and $\rho(\alpha) \delta_0$ to $\rho(\alpha\beta) \delta_0$. 

Here is one way of thinking about what the cocycle $\xi$ is measuring. 
\begin{align*}
    \xi(\gamma) &= \varphi( \gamma\cdot \delta_0) -  \rho(\gamma)\varphi(\delta_0) \\
              &=   \varphi( \delta_0) + \varphi( \gamma\cdot \delta_0) -\varphi( \delta_0) -  \rho(\gamma)\varphi(\delta_0)\\
              &= \Big( 1-\gamma \Big) \varphi( \delta_0) + \Big( \varphi( \gamma\cdot \delta_0) -\varphi( \delta_0) \Big) 
\end{align*}
Note that $\Big( 1- \gamma \Big) \varphi( \delta_0)$ is a coboundary. We are then left to consider $\Big( \varphi(\gamma\cdot \delta_0) -\varphi( \delta_0) \Big)$, which is the total relative motion from $\delta_0$ to $\gamma\cdot\delta_0$. The cocycle $\xi(\gamma)$ exactly assigns the total change in relative motion from some cell to a translate of that cell by $\gamma$. Ultimately, the final cohomology class does not rely on our choice of $\delta_0$.

We see that these $\xi$ are cocycles; however, a different choice of initial data $\delta_0$ and $\zeta_0$ (and thus a different $\varphi$), yields a cohomologous $\xi'$. 
Thus, up to a choice in initial data, we can discuss $[\xi] \in H^1(\Gamma, V)$.

In summary, $\xi: \Gamma \to V$ is a $\rho$-cocycle whose cohomology class depends only on the data of $\psi$. If we let $\Psi(\pm \Tilde{\mathscr{W}}, V)$ denote the vector space of maps $\psi:\pm \Tilde{\mathscr{W}} \to V$ satisfying the consistency conditions defined earlier, then we see that by the above arguments, we have the $\mathbb{R}$-linear map $E$ such that:
\begin{align*}
    E: \Psi(\pm\Tilde{\mathscr{W}},V) &\longrightarrow H^1(\Gamma, V)\\
    \psi &\longmapsto [\xi]
\end{align*}
To use this map to establish a lower-bound on the dimension of $H^1(\Gamma, V)$, we establish the injectivity of $E$. Throughout, recall that $Fix_{\Gamma}(\delta)$ denotes the set of elements in $\Gamma$ that fix a cell $\delta \in \mathscr{C}$ setwise
\begin{lem} \label{lem:faith}
    Let $E$ be as defined above. If $SO^+(n,1)$ acts faithfully on V, then $E$ is injective. 
\end{lem}
\begin{proof}
We proceed by assuming a $\psi$ map is mapped to the 0-cohomology class. By choosing $\zeta_0=0$ the associated $\varphi$ map is such that the cocycle $\xi$ is the 0-cocycle. 

As $\xi$ is assumed to be the 0-cocycle (and thus in the kernel of $E$), then we assert $\Big( \varphi( \gamma\cdot \delta) - \rho(\gamma)\varphi( \delta) \Big)=0$ for all $\gamma \in \Gamma$ and for all $\delta \in \mathscr{C}$.
%by \textcolor{red}{rem:zeta}. 
Hence, for all $\delta \in \mathscr{C}$ and $\gamma \in Fix_{\Gamma}(\delta)$, 
\[ \rho(\gamma)\varphi(\delta)=\varphi(\gamma\cdot \delta) = \varphi(\delta) \]
and so $\varphi(\delta)$ is fixed by $Fix_{\Gamma}(\delta)$. As we required that $Fix_{\Gamma}(\delta)$ is Zariski-dense in $SO^+(n,1)$ and that $SO^+(n,1)$ acts faithfully on $V$, then it follows that  $\varphi(\delta)=0$. This holds for all $\delta \in \mathscr{C}$, hence $\psi$ is identically zero, as desired. 
\end{proof}

\begin{remark*}
    To guarantee that $Fix_{\Gamma}(\delta_0)$ is Zariski-dense in $SO^+(n,1)$, it suffices that for each wall  $w \in \pm \Tilde{\mathscr{W}}$, the image of $\pi_1(w)$ under $\rho$ is Zariski dense in a copy of $SO^+(n-1,1)$, but not all the same copy. In fact, it is enough for each cell $\delta \in \tilde{\mathscr{C}}$ to have two walls with Zariski-dense image in distinct copies of $SO^+(n-1,1)$. Then we claim $Fix_{\Gamma}(\delta) \supset \langle Fix_{\Gamma}(w) \ |\  w \in \partial \delta \rangle$. To see this, we first notice that an element $s \in Fix_{\Gamma}(w)$ is also an element of $Fix_{\Gamma}(\delta)$. Then the product over elements of $Fix_{\Gamma}(w)$ ranging over $w\in \partial\delta$ must also fix $\delta$. The group generated by $\langle Fix_{\Gamma}(w) \ | w \in \partial \delta \rangle$ thus has Zariski closure a copy of $SO^+(n,1)$. In fact, it is enough for each cell $\delta \in \tilde{\mathscr{C}}$ to have two walls with Zariski-dense image in distinct copies of $SO^+(n-1,1)$ for identical reasons. Compare this to the condition that all the walls have non-elementary Fuchsian fundamental group in Theorem \ref{thm:bas}. 
\end{remark*}

We now have an injective map $E: \Psi(\pm\Tilde{\mathscr{W}},V) \to H^1(\Gamma, V)$. For a given $M$ and $\Delta$, it remains now to describe the family of all such $\psi \in \Psi$ supported on $\Delta$. 

\section{Establishing the General Form for the Lower Bound} \label{sec:low}
Let $M$ and $\Delta$ be as described previously. Recall that $\mathscr{W}$ is the collection of (totally geodesic) codimension-$1$ walls in $\Delta$, and $\Tilde{\mathscr{W}}$ is the preimage of $\mathscr{W}$ in $\Tilde{M}$. We aim to construct a map $\psi: \pm\Tilde{\mathscr{W}} \to V$ from the data of $\mathbb{R}$-valued weights on walls we shall define. When relevant, we will specify that  $V$ is either $\mathbb{R}^{n,1}$ or $\nu_{n+1}$. 

We consider for now deformations of $\rho$ into the Lie group $G$, where $G$ is either $SO^+(n+1,1)$ or $SL_{n+1}(\mathbb{R})$. Each $w \in \Tilde{\mathscr{W}}$ has a stabilizer $G_w$, which is a $1$-parameter subgroup in $G$. Infinitesimally then, we can consider the Lie subalgebra $\mathfrak{g}_w$, which is tangent to $G_w$. Since $\mathfrak{g}_w$ is also $1$-dimensional, we establish a choice of basis vector $v_w \in \mathfrak{g}_w$ for our convenience. More precisely, each of the $G_w = \left\{ \text{exp}(tv_w)\big|\  t \in \mathbb{R}, v_i \in \mathfrak{g}_{i}\right\}$ is a $1$-parameter subgroup of $G$, with parameter $t$. Then, we can say that $\frac{d}{dt} \text{exp}(tv_i)\Big|_{t=0}$ is an infinitesimal isometry we call $v_w$, which spans $\mathfrak{g}_w$.

When $G=SO^+(n+1,1)$, a wall $w$ lies in a totally geodesic copy of $\H{n-1}$, and thus $G_w \cong SO(2)$. Then $\mathfrak{g}_w\cong\mathfrak{so(2)}$. In an appropriate basis, we can normalize so that we can express $SO(2)$ in the following speed-$2\pi$ parametrization:
\[
G_w = 
\left\{
\left(
    \begin{NiceArray}{ccccc}
     \cos(t) & -\sin(t)  &    & &  \\
        \sin(t) & \cos(t) &  &  &      \\ 
        &&1 \Block[borders={top,left, tikz=dashed}]{3-3}{} && \\
          & &     & \ddots & \\
         &  &       &  & 1
    \end{NiceArray}
\right)
\ 
\middle\vert
\
t \in \mathbb{R}
\right\}
\]
Then, 
\[
v_w = \partial_tG_w\big|_{t=0}=
\left(
    \begin{NiceArray}{ccccc}
     0 & -1  &    & &  \\
        1 & 0 &  &  &      \\ 
        && \Block[borders={top,left, tikz=dashed}]{3-3}<\Large>{0} && \\
          & &     & \phantom{0000} & \\
         &  &       &  & 
    \end{NiceArray}
\right)
\]

We do a similar normalization when $G=SL_{n+1}(\mathbb{R})$. In this case, the stabilizer of a wall $G_w$ is a $1$-dimensional $\mathbb{R}$-Lie group of the form (in the appropriate basis): 
\[ 
G_w =
\left\{
\left(
    \begin{NiceArray}{cccc}
     e^{-nt} &   &    &   \\
             & e^{t}\Block[borders={top,left, tikz=dashed}]{3-3}{} &  &      \\
           &     & \ddots & \\
           &       &  & e^{t}
    \end{NiceArray}
 \right)
 \ 
\middle\vert
\
t \in \mathbb{R}
 \right\}
\]
which yields
\[ 
v_w = \partial_tG_w\big|_{t=0}=
\left(
    \begin{NiceArray}{cccc}
    -n &   &    &   \\
             & 1\Block[borders={top,left, tikz=dashed}]{3-3}{} &  &      \\
           &     & \ddots & \\
           &       &  & 1
    \end{NiceArray}
\right)
\]

For a choice of $G$, each of the $G_w$ and $\mathfrak{g}_w$ are canonically isomorphic. We have chosen a basis for the stabilizer of each $w\in\mathscr{W}$, and see that the isomorphism is induced by the appropriate (and unique) change of basis between these $G_w$. 

Finally, we make an arbitrary assignment of transverse orientation for each $w\in\mathscr{W}$. We define that if $v_w$ is the basis of the infinitesimal stabilizer $\mathfrak{g}_w$ for $w \in \pm \mathscr{W}$, then $-v_w$ is the basis of the infinitesimal stabilizer for $-w$.

With these conventions established, we describe the subspace of $\Psi(\pm\Tilde{\mathscr{W}}, V)$ that is supported by our branched bending complex $\Delta$, the image of the map $E$, which we henceforth denote $\Psi_\Delta$. To do this, we make use of a set $S(\pm\Tilde{\mathscr{W}}, V)$ which describes maps $\varsigma: \pm\Tilde{\mathscr{W}} \to V$ which only satisfy $\psi$-consistency conditions $1$ and $2$. We will describe the $\Psi_\Delta \subseteq \Psi(\pm\Tilde{\mathscr{W}}, V)$ as subset of $S(\pm\Tilde{\mathscr{W}}, V)$ via two maps: a map $F: \mathbb{R}^{|\mathscr{W|}} \to S(\pm\Tilde{\mathscr{W}}, V)$ and a map $H: S(\pm\Tilde{\mathscr{W}},V) \to \mathbb{R}^{A|\mathscr{B}|}$ so that $\Psi_\Delta = \text{im}(F)\cap\text{ker}(H) \subseteq \Psi(\pm\Tilde{\mathscr{W}}, V) \subseteq S(\pm\Tilde{\mathscr{W}}, V) $, where $A$ is a constant to be defined shortly. 

With this in mind, we can define the map
\[
% F: \mathbb{R}^{|\mathscr{W|}} \to \Psi(\pm\Tilde{\mathscr{W}}, V)
F: \mathbb{R}^{|\mathscr{W|}} \to S(\pm\Tilde{\mathscr{W}}, V)
\]
We can think of $F$ as an assignment of $\mathbb{R}$-valued ``weights'' which we call \emph{momenta} to each wall $w \in \mathscr{W}$. We construct a map $\varsigma: \pm\Tilde{\mathscr{W}} \to V$ from this data as follows. First, we say that for each $w \in \Tilde{\mathscr{W}}$, $\varsigma(w)\in \mathfrak{g}_w$. If $\omega_w\in\mathbb{R}$ is the value assigned by $F$ to $w \in \mathscr{W}$, then we establish that $\varsigma(\tilde{w})=\omega_w\cdot v_{\tilde{w}}$ for each lift $\tilde{w}$ of $w$. Next, following the definition of $\psi$, we observe that $\varsigma$ is $\rho$-equivariant. Finally, since $v_w$ respects a transverse orientation on $w$, we have the consistency condition $\varsigma(-w)=-\varsigma(w)$. 

We are now interested in the subset of Im($F$) containing maps that satisfy the consistency condition that the total motion around any vertex is zero; we will construct a map
\[ H: 
S(\pm\Tilde{\mathscr{W}},V) \to \mathbb{R}^{A|\mathscr{B}|}
\]
so that ker$(H)$ is the subset of $S(\pm\Tilde{\mathscr{W}},V)$ satisfying this consistency condition, for $A$ a constant to be defined later. To do this, consider each $b\in\mathscr{B}$ and choose an associated distinguished lift $\Tilde{b}\in\Tilde{\mathscr{B}}$. We consider the collection of walls $\{ w_i\} \in \Tilde{\mathscr{W}}$ meeting at $\Tilde{b}$. We consider an oriented loop around $\Tilde{b}$ assign the appropriate transverse orientation to each of the $\{w_i\}$, now in $\pm \Tilde{\mathscr{W}}$. 

Each such arrangement of a binding $\Tilde{b}$ and its incident collection of (oriented) walls $\{w_i\} \in \pm \Tilde{\mathscr{W}}$, gives rise to the condition
\[ \sum \varsigma(w_i)=0\]
A priori, this equation gives rise to dim($V$) many equations; however, we will show later that this can actually be reduced to $A$ equations, where $A <$ dim(Stab$_G(\H{n-2}))$). (For further discussion of the constant $A$, see the Guiding Principle in \ref{con:guide}.) %Thus, we have that ker$(H)$ is precisely the set described previously. 

We define the set $\Psi_\Delta =$ im$(F) \cap$ ker$(H)$ to be then the set of branched bending $\psi$ maps supported on $\Delta \in M$. Moreover, $\Psi_\Delta$ is 
a $\mathbb{R}$-subspace of $\Psi$ of dimension at least $({|\mathscr{W|}}-{A|\mathscr{B}|})$. By Lemma \ref{lem:faith}, $E(\Psi_\Delta)$ is a $\mathbb{R}$-subspace of $H^1(\Gamma, V)$ of dimension at least $({|\mathscr{W|}}-{A|\mathscr{B}|})$.

\section{Geometric Equations for Branched Bending} \label{sec:prod}
In this section we establish a geometric description of the equations defining the kernel of $H$ and thus, establishing the value of the constant $A$ for a given semisimple Lie group $G$. The framing of this discussion in terms of a ``product of matrices'' problem is inspired by the approach found in \cite{Kapovich2007-aa}.

In the general setting, we consider a binding $b \in \mathscr{B}$, and consider a distinguished lift $\Tilde{b} \in \Tilde{\mathscr{B}}$. As previous, we consider its incident collection of (oriented) walls $\{w_i\} \in \pm \Tilde{\mathscr{W}}$, which gives rise to the condition
\[ \sum \psi(w_i)=0\]
Let $G$ be a Lie group with a fixed collection of $1$-parameter subgroups $G_1, \dots, G_k \subset G$. As previous, $G_{w_i}$ is the stabilizer of $w_i \in \pm \tilde{\mathscr{W}}$. Then we define the map
\[
\text{Prod:} \prod^{k}_{i=1} G_{w_i} \rightarrow G
\]
given by Prod$(g_{w_1}, \dots, g_{w_k}) = g_{w_1} g_{w_2} \cdots g_{w_k}$ 
and see that the image is contained in $Stab_G(\tilde{b})$. We are interested in the ``kernel'' of this map; these are precisely the assignments of bending deformations along the walls $w_i$ that induce a deformation that does not change the monodromy around $b$.

%\subsection{Branched Higher Hyperbolic Bending}
\subsection{Proof of Theorem \ref{thm:bbson1}} \label{sec:hh}
Let $G=SO^+(n+1,1)$ and $G_w\cong SO(2)$ for each $w\in\mathscr{\Tilde{W}}$. The intersection of all $\{ w_i\}$ is $\tilde{b}$. The image of Prod then lands in $Stab_{G}(\H{n-2})\cong SO(3)$; we require that the image is in fact the identity to preserve the monodromy around the binding after deformation. Infinitesimally, this gives rise to an analogous problem: we aim to find the preimage of $0\in V$ under
\[
\text{Sum:} \bigoplus^{k}_{i=1} \mathfrak{g}_{w_i} \rightarrow V
\]
where $\mathfrak{g}_{w_i}$ is defined as previous. 

We choose coordinates on $\H{n+1}$ so that the infinitesimal stabilizer of the copy of $\H{n-2}$ containing the binding $\tilde{b}$ can be expressed as follows: 
\[
\mathrm{stab}_{\mathfrak{so}(n+1,1)}(\H{n-2}) = 
\left\{
    \begin{pNiceMatrix}[margin,columns-width=auto]
     \Block[borders={bottom,right}]{3-3}{\mathfrak{so}(3)} &   & \Block{3-4}<\Large>{0} &  &  &   \\
      &  &  &  &  &  \\
     &  &  &  &  &  \\
     \Block{4-3}<\Large>{0} & \phantom{0} &  & \Block[borders={top,left, tikz=dashed}]{3-3}<\Large>{0} &  &\\
    &   &  &  &  & \\
    &   &  &  &  & \\
    &   &  &  &  & 
    \end{pNiceMatrix}
 %\right)
 \right\}
\]
Since im(Sum) $\subset \mathfrak{so}(3)$, we can, for convenience, restrict our attention to the top-left block. 

\begin{figure}
    \centering
\includegraphics[width=0.2\linewidth]{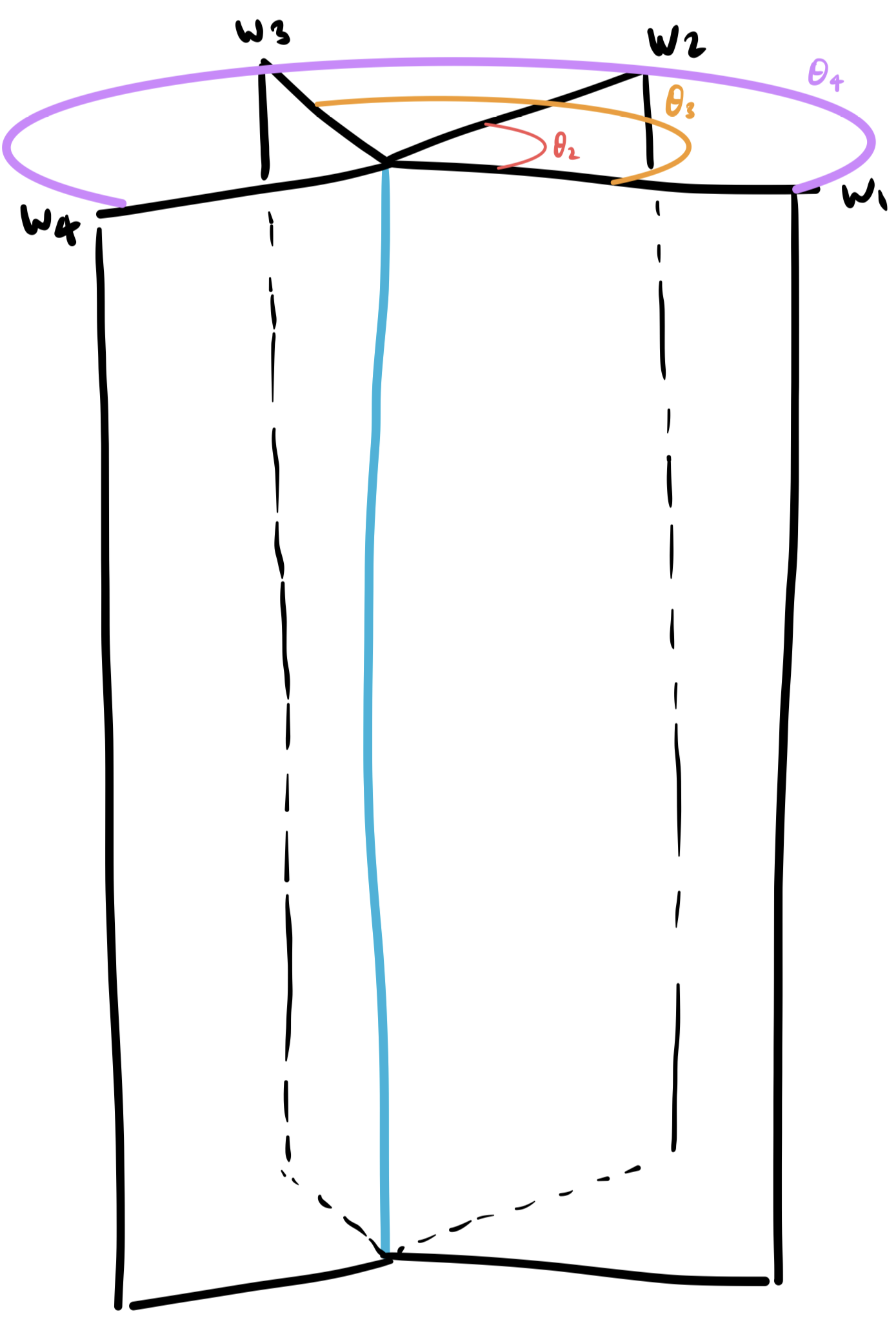}
    \caption{An arrangement of walls $w_i$ meeting around a binding}
    \label{fig:pages}
\end{figure}

Recall the map $F:  \mathbb{R}^{|\mathscr{W|}} \to \Psi(\pm\Tilde{\mathscr{W}},V)$ as defined previously. Elements $\psi \in \Psi(\pm\Tilde{\mathscr{W}},V)$ arise from an assignment of $\mathbb{R}$-valued momenta $\omega_w$ to each wall $w \in \mathscr{W}$. In particular, each $\psi$ assigns the value of $\omega_w \cdot v_w$ to each 
%(positively transversely oriented) 
wall $w \in \pm\mathscr{W}$ and its lifts. However, in our computations for $G=SO(n+1,1)$, the orientation data can be forgotten; for more on this, see Remark \ref{rem:why}. 

We now construct the set of maps of this form that are also in ker($H$); that is, for each $b$ and its incident and transversely oriented $\{ w_i \}$ (around some distinguished lift $\tilde{b}$), we obtain
\[
\sum^k_{i=1} \omega_{w_i}\cdot v_{w_i} = 0
\]

For each such collection of $\{w_i\}$ we can make a convenient calculation. If we choose $w_1$ with stabilizer $\mathfrak{g}_{w_1}$ as the initial wall in our enumeration, then we can describe the appropriate conjugation to the stabilizer $\mathfrak{g}_{w_i}$ of another wall $w_i$ incident to $\tilde{b}$ as follows: 
\[  v_{w_1} = 
\begin{pmatrix}
0&-1&0\\
1& 0 & 0  \\
0& 0 & 0
\end{pmatrix}, \qquad 
 _{R_{\theta_i}}(v_{w_1}) = v_{w_i}
\begin{pmatrix}
0&-\omega  \cos(\theta_i)&\omega  \sin(\theta_i)\\
\omega  \cos(\theta_i)& 0 & 0 \\
-\omega  \sin(\theta_i)&0 & 0
\end{pmatrix}
\]
where $\theta_i$ is the angle formed by $w_1$ and $w_i$ meeting around $\tilde{b}$, and $R_{\theta_i} \in SO(2)$ is the rotation around $\tilde{b}$ by $\theta_i$. From this, we rewrite the above to: 
\[
\omega_{w_1}+ \sum^k_{i=2} \omega_{w_i}\cdot  {R_{\theta_i}} v_{w_1} = 0
\]
%\[ \omega_1 \cdot v_1 +  _{R_{\theta_2}}(\omega_2 \cdot v_1) + \cdots+ _{R_{\theta_k}}(\omega_k \cdot v_1) = 0   \]
where $\theta_2, \dots,\theta_k$ are the angles between the hyperplanes relative to the first indexed hyperplane. There are two distinct entries in these matrices; we see this gives rise to two $\mathbb{R}$-valued equations: 
\begin{equation} \label{eq:eq1}
    \omega_{w_1} + \omega_{w_2}\cos(\theta_2) + \cdots + \omega_{w_k}\cos(\theta_k) = 0 \\
\end{equation}
\begin{equation} \label{eq:eq2}
    \omega_{w_2}\sin(\theta_2) + \cdots + \omega_{w_k}\sin(\theta_k) = 0 \\
\end{equation}
Thus, $A=2$. The
nontrivial assignments of $\omega_{w_i}$ that satisfy these 2 equations are precisely those that land in im$(F)\cap$ker$(H)$. We solve this system once for each $b\in\mathscr{B}$. Thus, since $c_{n-1}=|\mathscr{W}|$, $c_{n-2}=|\mathscr{B}|$, and  $A=2$, we establish:

\begin{theorem}[Theorem \ref{thm:bbson1}] \label{thm:bbson1fin}
Suppose $M$ = $\Gamma \backslash \mathbb{H}^n$ is a complete hyperbolic $n$-manifold containing a branched totally geodesic hypersurface $\Delta$ with branch locus $\mathscr{B}$. Let $c_{n-1}$ be the number of $(n-1)$-dimensional complementary regions in $\Delta\backslash B$, and let $c_{n-2}$ be the number of components of $\mathscr{B}$. Furthermore, let $\Tilde{M}\backslash \Tilde{\Delta}$ be composed of top-dimensional regions $\mathscr{C}$, each with Zariski-dense Fix$_{\Gamma}(\delta)$, $\delta \in \mathscr{C}$. 
Then $dim_{\mathbb{R}} H^{1}(M, \mathbb{R}^{n,1}) \geq c_{n-1} - 2c_{n-2}$. 
\end{theorem}

We observe the following special case. 
\begin{lem}
	Consider a branched bending system where every bending locus is of degree $n$, and where the angle between two neighboring prongs is $\frac{2\pi}{n}$. Then a solution to this branched bending system is given by all $\omega_i$ equal. 
\end{lem}

\begin{proof} \label{lem:even}
	By the equations defined above, we see that for an $n$-valent branched bending locus, we are verifying solutions to the following equations when all $w_i$ are equal and $\theta = 2\pi/n$: 
\[ 
	\sum_{k=0}^{n-1} \cos(k\theta)   =  0, \qquad
    \sum_{k=0}^{n-1} \sin(k\theta)   =  0
\]
This reduces to checking that the sum of $n^{th}$-roots of unity is zero, which is a classical result.
\end{proof}
 
\subsection{Proof of Theorem \ref{thm:bbpgln}}\label{sec:rp}
Here we find analogues to Equations \ref{eq:eq1} and \ref{eq:eq2}, but in the setting of projective deformations. See Section \ref{prelim} for more discussion on projective bending. 

We work in the same setting as before: we lift a triple $M$, $\mathscr{W}$,  and $\mathscr{B}$ to their respective universal covers $\Tilde{M}$, $\Tilde{\mathscr{W}}$, and $\Tilde{\mathscr{B}}$; for a binding $b \in \mathscr{B}$, consider a chosen lift $\Tilde{b} \in \Tilde{\mathscr{B}}$, and look at all of the $w_i \in \Tilde{\mathscr{W}}$ incident to $\Tilde{b}$. 

As we are interested in projective deformations, recall that $\Tilde{M}=\H{n}\subset \mathbb{RP}^n$. Each lift of a wall $w_i$ is contained in a copy of $\H{n-1}$; we call this collection $\{ \H{n-1}_i \}$, and note that $\Tilde{b}$ lifts to a distinguished copy of $\H{n-2}$ such that $\cap_i \H{n-1}_i = \H{n-2}$. Each of the copies of $\H{n-1}$ has a $1$-dimensional family of \emph{bulging} deformations. 

In a bulging deformation, a totally geodesic hyperplane $\H{n-1} \subset\H{n} \subset \mathbb{RP}^{n}$ has an orthogonal complement  $(\H{n-1})^{\perp} \subset \mathbb{RP}^n$ that determines the direction of the deformation. For each of the $\H{n-1}_{i}$, there is a corresponding $(\H{n-1}_{i})^{\perp}$. Moreover, as $\H{n-2} \subseteq \cap_i \H{n-1}_{i}$, each $(\H{n-1}_{i})^{\perp}$ lies in $(\H{n-2})^{\perp}$, as seen in Figure \ref{fig:projbulg}. 

\begin{figure}
    \centering
    \includegraphics[width=0.5\linewidth]{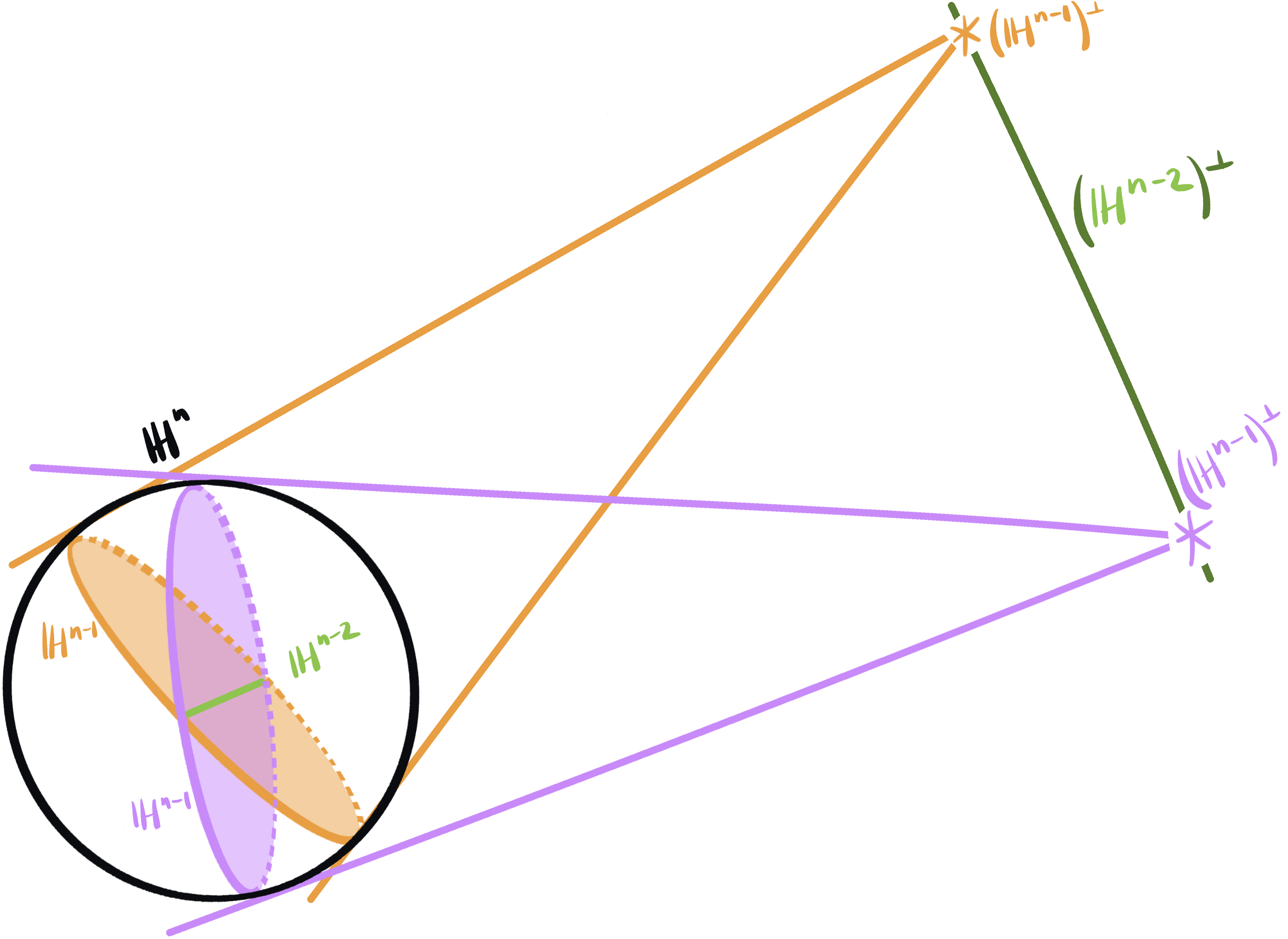}
    \caption{Two codimension-1 hyperbolic spaces meeting along a codimension-2 hyperbolic space, and their respective orthogonal complements in $\mathbb{RP}^n$}
    \label{fig:projbulg}
\end{figure}

In analogy to the construction in Section \ref{sec:hh}, the product of bulging deformations supported along each of the $\H{n-1}_{i}$ must be contained in the subgroup of Isom$(\mathbb{RP}^{n})$ preserving $(\H{n-2})^{\perp}$. We work now in a convenient basis to to explore this subgroup. 

We pick a copy $\H{n-1}$ to do bulging along, with coordinates chosen so that 
\[ 
Stab_{Isom(\mathbb{RP}^n)}(\H{n-1}) =
\left\{
\left(
    \begin{NiceArray}{ccccc}
     e^{-nt} &   &    & &  \\
     & e^{t}\Block[borders={top,left, tikz=dotted}]{4-4}{} & & & \\
            & & e^{t}\Block[borders={top,left, tikz=dashed}]{3-3}{} &  &\\
          & &     & \ddots & \\
         &  &       &  & e^{t}
    \end{NiceArray}
 \right)
 \ 
\middle\vert
\
t \in \mathbb{R}
 \right\}
 \]
Here, the dashed $(n-1) \times (n-1)$ block still corresponds to acting on $\H{n-2}$. The larger $n \times n$-block acts on our chosen copy of $\H{n-1}$ by something projectively equivalent to the identity; and the top-leftmost entry corresponds to $(\H{n-1})^{\perp}$. In these coordinates, a change in choice of $\H{n-1}$ still fixes the bottom-right dashed $(n-1) \times (n-1)$ block; only the top-left $2 \times 2$ block can change. So we see then that in the case of projective bending, im(Prod) $\subset GL(2)$. 

Again, looking to the infinitesimal version of this setting, we see that im(Sum) $\subset \mathfrak{gl}(2)$. In particular, we choose coordinates on $\mathbb{RP}^n$ so that
\[
\text{im(Sum)} \subset
\left\{
    \begin{pNiceMatrix}[margin,columns-width=auto]
     \Block[borders={bottom,right}]{2-2}{\mathfrak{gl}(2)} &   & \Block{2-4}<\Large>{0} &  &  &   \\
            &  &  &  &  &  \\
           \Block{4-2}<\Large>{0} & \phantom{00}  & \Block[borders={top,left, tikz=dashed}]{4-4}{\ast I} &  &  &\\
           &   &  &  &  & \\
           &   &  &  &  & \\
           &   &  &  &  & 
    \end{pNiceMatrix}
    \middle\vert \ 
    \text{trace} =0
 \right\}
\]
We now restrict our attention to the top-left block for the following discussion. This means that if we write 
\[ v_w =\begin{pmatrix}
-n &&&\\
&1&&\\
&&\ddots&\\
&&&1
\end{pmatrix}\]
then restricted to the top-left block we obtain
\[ 
v_w = 
\begin{pmatrix}
    -n &0\\
    0 &1
\end{pmatrix}\]
where $n$ is the dimension of our manifold. 

As previous, we aim to construct the set of maps in im$(F)\cap$ ker($H$); that is, for each $b$ and its incident and transversely oriented $\{ w_i \}$ (around some distinguished lift $\tilde{b}$), we obtain
\[
\sum^k_{i=1} \omega_{w_i}\cdot v_{w_i} = 0
\]

For each such collection of $\{w_i\}$ we make a similarly convenient calculation. In this setting, we emphasize that the transverse orientation of each wall $w_i$ is taken into account in the calculations; the $\omega_{w_i}$ are computed with sign. If we choose $w_1$ with stabilizer $\mathfrak{g}_{w_1}$ as the initial wall in our enumeration, then we can describe the appropriate conjugation to the stabilizer $\mathfrak{g}_{w_i}$ of another wall $w_i$ incident to $\tilde{b}$ as follows. 
Let $\alpha = -\dfrac{1}{2}n + \dfrac{1}{2}$, $\beta = \dfrac{1}{2}n + \dfrac{1}{2}$, and $n$ be the dimension as it has been throughout; then 
\[ v_1 = 
\begin{pmatrix}
 - n & 0  \\
0 & 1
\end{pmatrix}, \qquad 
 _{R_{\theta}}( v_1) = 
\begin{pmatrix}
(\alpha-\beta\cos(2\theta)) & \beta \sin(2\theta)\\
\beta \sin(2\theta) & (\alpha+\beta\cos(2\theta))
\end{pmatrix}
\]
where $\theta_i$ is the angle formed by $w_1$ and $w_i$ meeting around $\tilde{b}$, and $R_{\theta_i} \in SO(2)$ is the rotation around $\tilde{b}$ by $\theta_i$. From this, we rewrite the above to: 
\[
\omega_{w_1}+ \sum^k_{i=2} \omega_{w_i}\cdot  {R_{\theta_i}} v_{1} = 0
\]
%\[ \omega_1 \cdot v_1 +  _{R_{\theta_2}}(\omega_2 \cdot v_1) + \cdots+ _{R_{\theta_k}}(\omega_k \cdot v_1) = 0   \]
where $\theta_2, \dots,\theta_k$ are the angles between the hyperplanes relative to the `first indexed hyperplane. There are three distinct entries in these matrices; this gives rise to \emph{three} $\mathbb{R}$-valued equations: 
\begin{equation} \label{peq1}
    \omega_{w_1} + \left( \sum^k_{i=2} \omega_{w_i}\left(\alpha+ \beta\cos(2\theta_i)\right) \right) =0
\end{equation}
\begin{equation} \label{peq2}
    -\omega_{w_1}n + \left( \sum^k_{i=2} \omega_{w_1}\left(\alpha- \beta \cos(2\theta_i)\right) \right) = 0
\end{equation}
\begin{equation} \label{peq3}
     \beta\left(\sum^k_{i=2} \omega_{w_1}\left(\sin(2\theta_i)\right) \right)=0 
\end{equation}
The nontrivial assignments of $\omega_{w_i}$ that satisfy these three equations are precisely those that land in im$(F)\cap$ker$(H)$. We solve this system once for each $b\in\mathscr{B}$. Thus, since $c_{n-1}=|\mathscr{W}|$, $c_{n-2}=|\mathscr{B}|$, and  $A=3$, we obtain:

\begin{theorem}[Theorem \ref{thm:bbpgln}] \label{thm:bbpglnfin}
Let $M$, $\Delta$, and $\mathscr{B}$, be as in Theorem \ref{thm:bbson1}.
Then $dim_{\mathbb{R}} H^{1}(M, \nu_{n+1}) \geq c_{(n-1)} - 3c_{(n-2)}$.
\end{theorem} 

\begin{figure}
    \centering
    \includegraphics[width=0.5\linewidth]{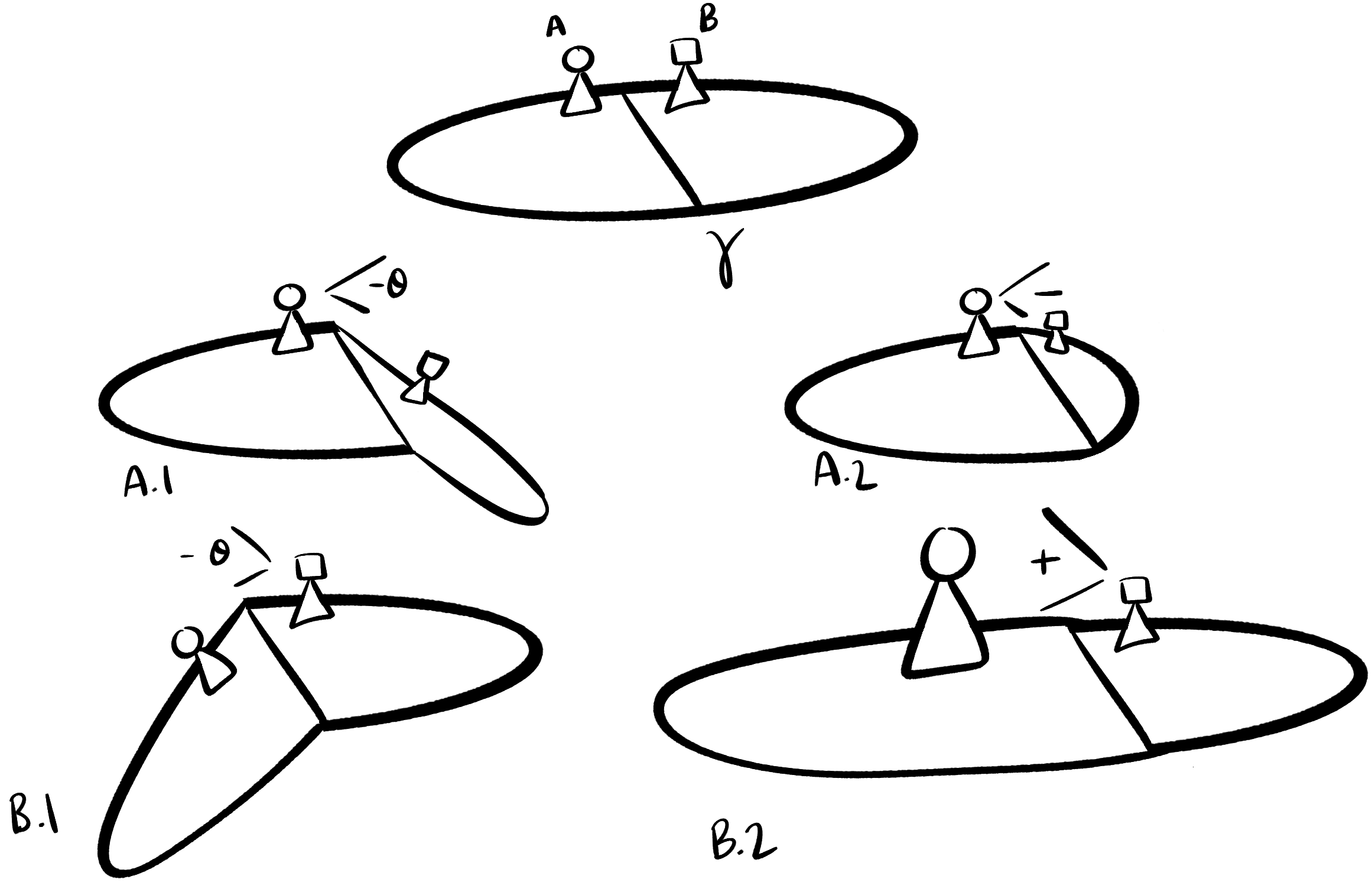}
    \caption{A schematic of bending and bulging as experienced by two observers}
    \label{rem:whyfig}
\end{figure}
 
\begin{remark*} \label{rem:why}
    One major discrepancy in the classical bending setting versus the projective bulging setting is the necessity of a \emph{transverse orientation} on the totally geodesic hypersurface. This is to account for an asymmetry that appears upon performing these deformations, as illustrated in Figure \ref{rem:whyfig}. In the classical bending case, two observers facing each other across the bending locus will ``see'' the same thing after the bending is performed; it will appear to either observer that the other observer has shifted away or towards them by the same angle $\theta$. However, in the case of projective bulging, the observers will ``see'' different phenomena, depending on which side of the bulging locus they are on. One observer will see the other dilated by a factor of $e^t$, while the other will witness a contraction by a factor of $e^{-t}$. To account for which side is being dilated or contracted, we assign a transverse orientation to the bulging locus. 
\end{remark*}

\subsection{Broader Application}

With the above arguments in mind, one may propose the following: 
\begin{guiding} \label{con:guide}
Let $M$, $\Delta$, and $\mathscr{B}$, be as in Theorems \ref{thm:bbson1fin}, \ref{thm:bbpglnfin} and let $G$ be a semisimple Lie group with Lie algebra $\mathfrak{g}$ such that $SO^+(n,1) \subset G$. Then 
\[dim_{\mathbb{R}} H^{1}(M, \mathfrak{g}_{Ad_{\rho}}) \geq c_{n-1} - Ac_{n-2}\]
where $A$ is a constant informed by the dimension of $Stab_{G}(\H{n-2})$.
\end{guiding}

In the case of \cite{Bart2006-im}, $n=3$, $G=SO^+(4,1)$ and $A=2$. We note that when $G=SO^+(n+1,1)$, then $A=2$ as it is in the case proved in Theorem \ref{thm:bbson1}; when $G=SL(n+1)$, then $A=3$ as it is in the case proved in Theorem \ref{thm:bbpgln}. 

We emphasize that $A$ is \emph{not} the dimension of $Stab_{G}(\H{n-2})$; in the $Ad$-invariant splitting of $\gad = \son{n} \oplus \nu$, where $\nu$ is the orthogonal complement of $\son{n}$ in $\gad$ with respect to the Killing form, $A$ can be less than even the dimension of the subspace of $\gad$ tangent to $Stab_{G}(\H{n-2})$, as seen in the case when $G=SL_{n+1}(\mathbb{R})$. 

In the general setting, the constant $A$ is the dimension of the space of all possible bending deformations into the chosen group $G$, along any $\H{n-1}$ containing a fixed copy of $\H{n-2}$. This is the space being described in the computations of ker$(H)$ in Sections \ref{sec:hh} and \ref{sec:rp}.

\section{The Deformation Space of the Borromean Rings} \label{sec:defbor}

So far, we have developed general machinery to describe how branched and intersecting hypersurfaces in finite-volume hyperbolic manifolds of dimension $\geq3$ are able to support deformations into certain geometries. In contrast to the broader setting of those results, in this section we focus our attention to branched bending in the setting of cusped hyperbolic $3$-manifolds. As an example, we analyze the Borromean rings complement, a thrice-cusped finite-volume hyperbolic $3$-manifold that is known not to admit any closed embedded totally geodesic hypersurfaces \cite{menasco1992totally}. 

\subsection{Cuspidal Cohomology} 

For this section, we let $\Gamma$ be a lattice in $SO^+(3,1)$. We are interested especially in the case where $M=\Gamma\backslash\H{3}$, $M$ a hyperbolic 3-manifold with cusps, and where $\rho: \Gamma \to SO^+(3,1)$ is a discrete and faithful representation. Throughout $G$ is a semisimple Lie group such that $SO(3,1) \subseteq G$.

The space of infinitesimal deformations of the representation $\rho$ into the Lie group $G$ with Lie algebra $\mathfrak{g}$ is identified $H^{1}(\Gamma, \mathfrak{g}_{Ad_{\rho}})$. We will define \emph{cuspidal cohomology} as the subspace of these infinitesimal deformations of $\rho$ that do not deform the cusps. To make this more precise:
\[ H_{\partial}^{1}(\Gamma, \mathfrak{g}_{Ad_{\rho}}) := \ker \Big(\text{res:}\  H^{1}((\Gamma\backslash\H{3}), \mathfrak{g}_{Ad_{\rho}}) \to H^{1}(\partial(\Gamma\backslash\H{3}), \mathfrak{g}_{Ad_{\rho}}) \Big) \]
This means that if $c$ is a cocycle in $H_{\partial}^{1}(\Gamma, \mathfrak{g}_{Ad_{\rho}})$, the restriction of $c$ to the peripheral subgroups should be trivial; that is, $c(\gamma)$ should act as a coboundary on parabolic elements. 

As discussed in Section \ref{prelim}, when $G=SO^+(4,1)$, we have that $\mathfrak{g}=\mathfrak{so}(4,1)$, and 
$ \mathfrak{so}(4,1)= \mathfrak{so}(3,1)\oplus \mathbb{R}^{3,1}$
which induces the splitting 
\[ H^{1}(\Gamma, \mathfrak{so}(4,1)) = H^{1}(\Gamma, \mathfrak{so}(3,1)) \oplus H^{1}(\Gamma, \mathbb{R}^{3,1})  \]
as well as a similar splitting in cuspidal cohomology. 

When $G=SL_{4}(\mathbb{R})$, we have that $\mathfrak{g}=\mathfrak{sl}_{4}(\mathbb{R})$, and 
$ \mathfrak{sl}_{4}(\mathbb{R}) = \mathfrak{so}(3,1) \oplus \nu_{4}$. Here, $\nu_{4}$ is a $9$-dimensional vector space. This also induces the splitting 
\[ H^{1}(\Gamma, \mathfrak{so}(4,1)) = H^{1}(\Gamma, \mathfrak{so}(3,1)) \oplus H^{1}(\Gamma, \nu_{4})  \]
as well as a similar splitting in cuspidal cohomology. 

We remark briefly here that in \cite{Bart2006-im}, in order to compute cuspidal cohomology, the authors compute the similar (but not a priori same) notion of \emph{parabolic cohomology}. It is a result of \citet{kapovich1994deformations} that these notions agree in the setting where $\Gamma$ is a lattice in $SO^+(3,1)$ and $\mathfrak{g}=\mathfrak{so}(4,1)$.
For a strengthening of this result and further discussion of these ideas, please see Appendix \ref{app:cusp}. We make use of this identification, and write $PH^{1}(\Gamma, \mathfrak{g}_{Ad_{\rho}})$ instead of $H_{\partial}^{1}(\Gamma, \mathfrak{g}_{Ad_{\rho}})$ from here on. 

By Calabi-Weil Infinitesimal Rigidity \cite{calabi1961compact, Weil1960-xi}, $PH^1(\Gamma, \mathfrak{so}(3,1))=0$ when $\Gamma$ is a uniform lattice in $SO^+(3,1)$. Due to Garland-Raghunathan  \cite{Garland1970-dy}, $PH^1(\Gamma, \mathfrak{so}(3,1))=0$ in the case that $\Gamma$ is a non-cocompact lattice in $SO^+(3,1)$. These results imply that to compute $H^{1}(\Gamma, \mathfrak{so}(4,1))$ and $H^{1}(\Gamma, \mathfrak{sl}_4(\mathbb{R})$ (and their cuspidal counterparts) reduces to computing  $H^{1}(\Gamma, \mathbb{R}^{3,1})$ and  $H^{1}(\Gamma, \nu_{4})$ respectively (with the corresponding results in cuspidal cohomology). 

\subsection{The Borromean Rings complement}

\begin{figure}
\centering
\begin{subfigure}{.49\textwidth}
  \centering
  \includegraphics[width=.49\linewidth]{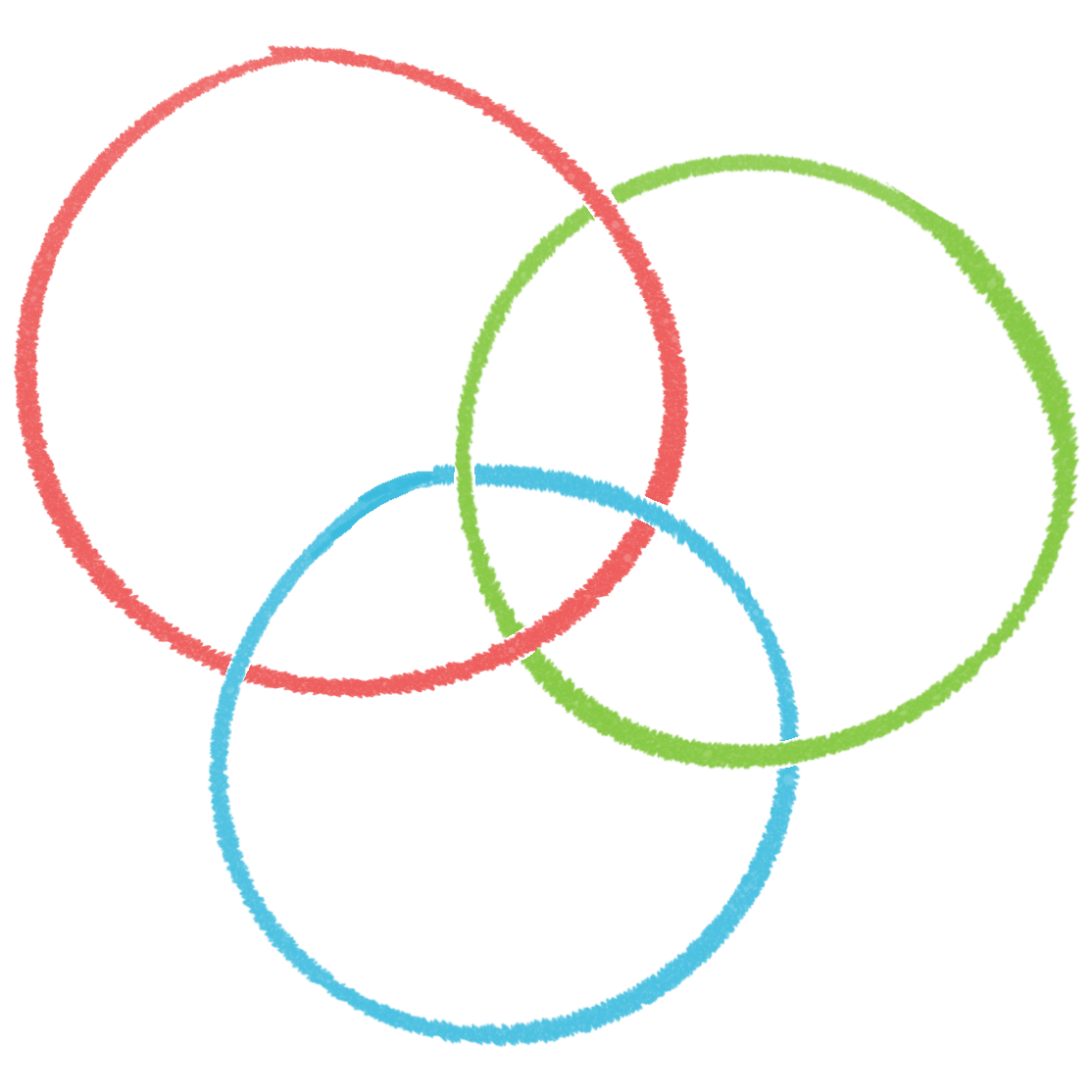}
  \caption{The link diagram}
  \label{fig:borroplain}
\end{subfigure}
\begin{subfigure}{.49\textwidth}
  \centering
  \includegraphics[width=.49\linewidth]{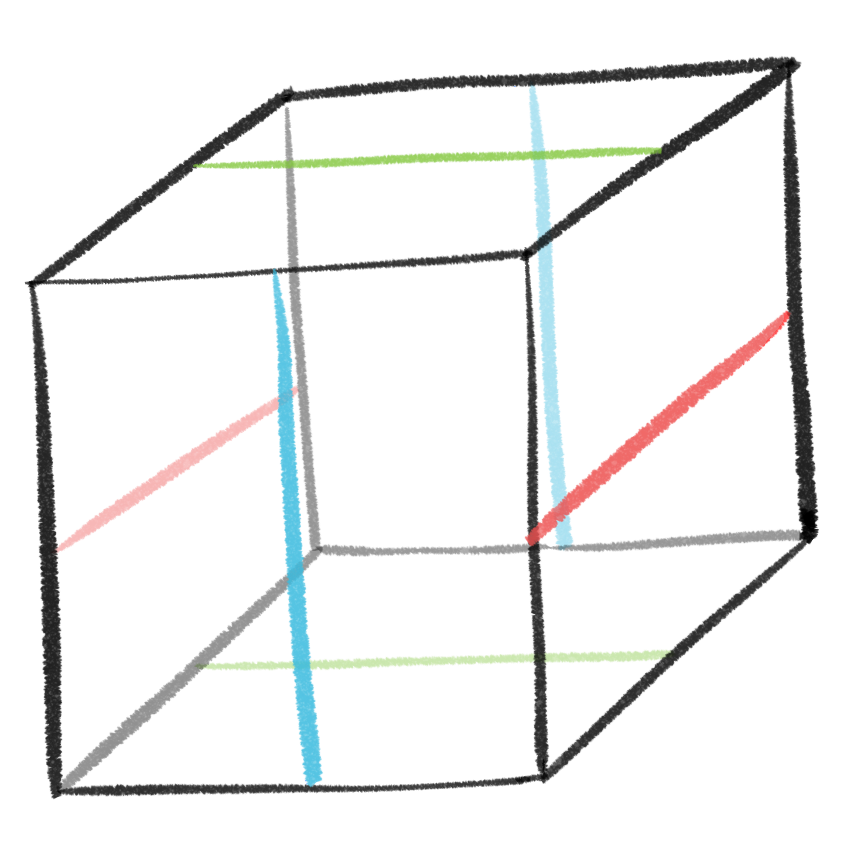}
  \caption{The cube diagram}
  \label{fig:cubeplain}
\end{subfigure}
\caption{Two visualizations of $B$}
\label{fig:borropics}
\end{figure}

\label{sec:borro}
We analyze the cuspidal cohomology of the Borromean rings complement in two different settings. First, with coefficients in $\mathfrak{so}(4,1)$, we show that the Borromean rings complement has trivial cuspidal cohomology. This gives an alternate proof that this link complement does not contain any closed embedded totally geodesic hypersurfaces, as shown in \cite{menasco1992totally}. Then, with coefficients in $\mathfrak{sl}_4(\mathbb{R})$, we show that the cuspidal cohomology associated to infinitesimal $SL_{4}(\mathbb{R})$ deformations of the Borromean rings complement does \emph{not} vanish, as one might expect from the lack of closed totally geodesic surfaces. Instead, we show this manifold contains a non-compact totally geodesic branched complex that supports deformations into higher rank that do not affect the cusps. (In some sense, we can approximate the effects of \emph{closed} totally geodesic surfaces by this system of non-compact ones.)

First, let $B$ denote the Borromean rings as seen in Figure \ref{fig:borroplain}, otherwise known as the link $6^{3}_{2}$ in \citet{rolfsen2003knots}. The link $B$ has three unknotted components and the property that any two components are unlinked. The complement of $B$ in $S^3$ is a thrice-cusped hyperbolic arithmetic manifold that we denote as $\borro$ from here on. Let $\Gamma = \pi_1(\borro)$, which admits the following presentation:  \[ \Gamma = \langle x,y,z\ |\  [x,  [y^{-1}, z]]= [y, [z^{-1}, x]] \rangle \] 
where x, y, and z, are a meridional generator for each of the three cusps in $M$. In particular, they are all parabolic elements.

A link that arises as the closure of a three-braid does not admit any closed embedded totally geodesic surfaces \cite{lozano1985incompressible, menasco1992totally}, and the Borromean rings can be realized as the closure of the three-braid word ($\sigma_{1}$$\sigma^{-1}_{2})^{3}$. However, it is possible to find several non-compact embedded totally geodesic surfaces. Of interest to us are six totally geodesic thrice-punctured spheres that arise as the fixed point set of involutions of $\borro$. These surfaces intersect each other either in right angles or not at all---we elaborate on their properties here. 

\begin{figure}
    \centering
    \includegraphics[width=0.5\linewidth]{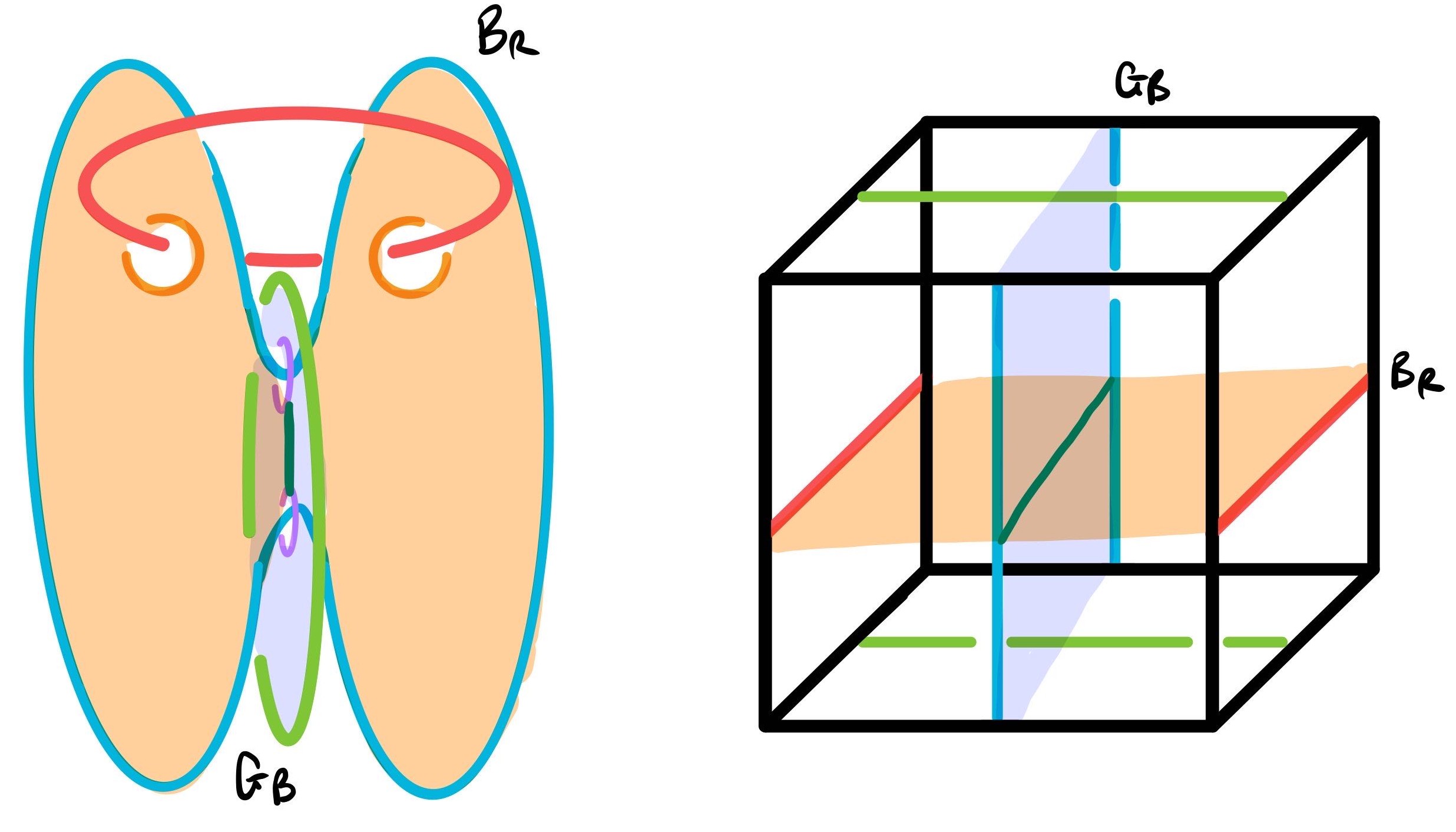}
    \caption{Two intersecting totally geodesic pairs of pants, represented in two models}
    \label{fig:intersection}
\end{figure}

There are two useful models when discussing the presence of these totally geodesic pairs of pants in the Borromean rings complement. In the link diagram, we consider a disk bounded in the plane by one of the components of the link $B$, and two punctures (or the ``cuffs'') arising from considering the passing of one of the other components in the link diagram through that disk. We remark that for each component of the link, there are two such disks (one for each of the remaining components of the link that do not form the waist). Said another way, when considering the diagram in Figure \ref{fig:intersection}, for each component of $B$, there are two disjoint thrice-punctured spheres. In total, we obtain six thrice-punctured spheres. A straightening lemma of \citet{Adams1985-wj} shows that each of these surfaces are isotopic to a totally geodesic thrice-punctured sphere. In Figure \ref{fig:intersection}, we see (after an isotopy of the link diagram) the intersection of two such surfaces. 

\begin{figure}
    \centering
    \includegraphics[width=0.5\linewidth]{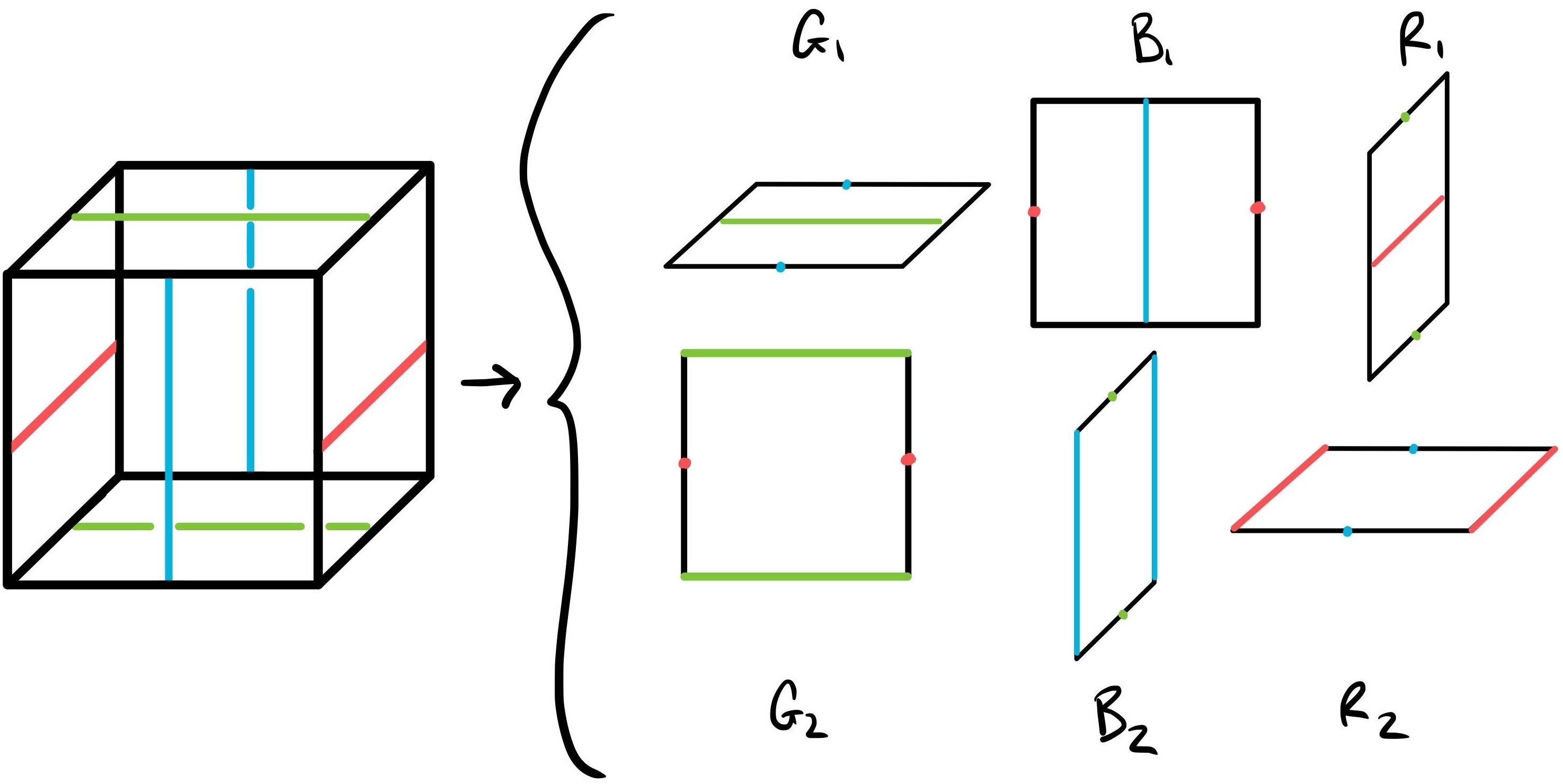}
    \caption{The six totally geodesic pairs of pants in the Borromean rings complement of interest; the top row comes from the faces, the bottom row comes from the midcubes}
    \label{fig:pantcube}
\end{figure}

However, via another perspective, we can visualize $\borro$ more directly. Let $[0,1]^{3}$ be the standard 3-cube, with opposite faces identified to form the standard 3-torus. If the arcs on the faces of the cube, as shown in Figure \ref{fig:cubeplain}, are deleted, then we obtain a manifold that is diffeomorphic to $\borro$ (this is stated in \cite{Hodgson_1986}). Moreover, we can also see the pairs of pants described previously here as the faces and midcubes of this cube, as seen in Figure \ref{fig:pantcube}. As there are involutions of this cube that fix these faces and midcubes, we have another way of seeing that the six pairs of pants described are totally geodesic. In this model, we can also see their incidences; in Figure \ref{fig:intersection}, we can see the same intersection of surfaces as in the link diagram at left also being represented in the cube diagram. Their incidences can be visualized in Figure \ref{fig:incidence}, where every intersection is right-angled. These six embedded totally geodesic pairs of pants form a complex of totally geodesic pieces that will be referred to in whole from here as $\pantsc$. For more on this construction, see \cite{Hodgson_1986}.

There are three cusps, $C_R$, $C_G$, and $C_B$, which correspond to tubular neighborhoods around the red, green, and blue components in the diagram of $B$ respectively. Each cusp $C_{i}$ has a torus cross-section $T_{i}$ such that $C_{i} = T_{i} \times [1, \infty)$, and each $T_{i}$ is a Euclidean 2-torus. When considering a thrice-punctured sphere, we will color the cusp ends based on which cusp in $\borro$ it intersects with for our convenience.

In Figure \ref{fig:incidence}, we see that for each cusp $C_{i}$, there are exactly four totally geodesic thrice-punctured spheres that intersect our chosen cusp. Let $\mathcal{P}_{m,n}$ refer to the pair of pants\footnote{Throughout, the author uses ``pair of pants'' interchangeably with ``thrice-punctured sphere"--a convention that is technically inaccurate. This is done to employ the evocative language of ``waistbands'' and ``cuffs'' when describing critical features of these surfaces. We hope that this does not cause confusion.} with $m$-colored waistband and $n$-colored cuffs. 
Here, the ``waistband'' of a given pair of pants refers to the puncture that meets a cusp distinct from the cusp the other two punctures (the ``cuffs'') meet. If we consider the intersection of $C_{i}$ with one such of the $\mathcal{P}_{m,n}$ at a time, then there are three cases:  $C_{i} \cap \mathcal{P}_{m,n}$ is empty, $C_{i} \cap \mathcal{P}_{m,n}$ is a meridian for $T_{i}$ and so $m=i$ , or $C_{i} \cap \mathcal{P}_{m,n}$ is a pair of disjoint but parallel longitudes for $T_{i}$ and so $n=i$. For a given $C_{i}$, there are two totally geodesic thrice-punctured spheres in each of these sets, corresponding to $m,n \in \{ R,G,B\}\backslash\{i\}$.

\begin{figure}
    \centering
    \includegraphics[width=0.5\linewidth]{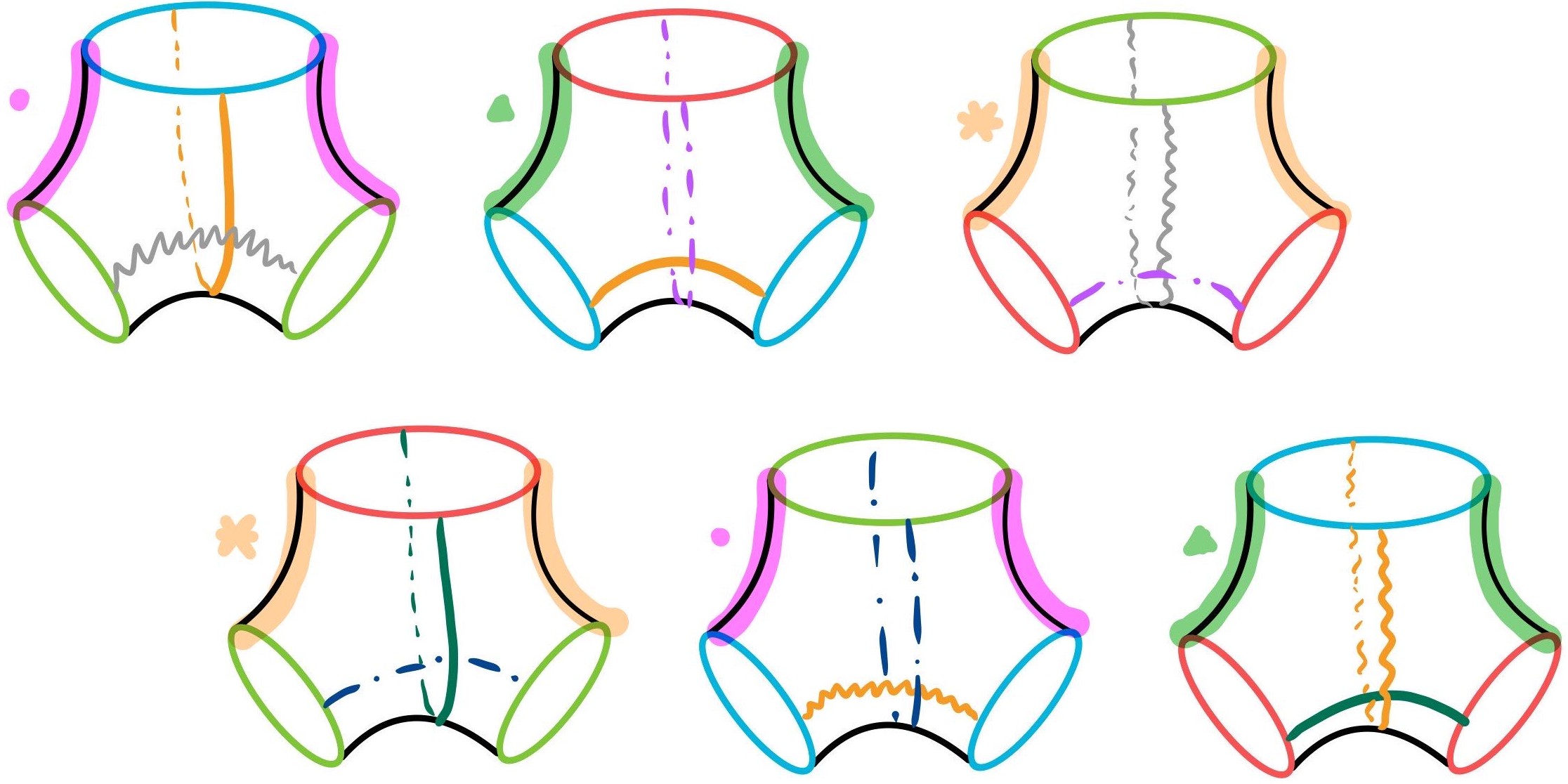}
    \caption{The six totally geodesic thrice punctured spheres, annotated by intersection}
    \label{fig:incidence}
\end{figure}

In the following two sections, we will be computing the dimension of deformation space associated to the complete hyperbolic structure on the Borromean rings complement via Fox Calculus computations.  For background on Fox Calculus, see \cite{Fox1953-yc}, \cite{Goldman2020-ym}; to see the implementation specific to this paper, please see Appendix \ref{app:fox}. To make these computations, we make use of the following representation $\rho$ of $\Gamma$ in $SO^+(3,1)$, adapted from a previous representation into $PSL_{2}(\mathbb{C})$ from \cite{Ucan-Puc2021-ra}:

\begin{equation} \label{eq:so31} 
\rho(x)= \begin{pmatrix}
3 & 0 & 2 & 2 \\
0 & 1 & 0 & 0 \\
-2 & 0 & -1 & -2 \\
2 & 0 & 2 & 1 
\end{pmatrix},
\ 
\rho(y)= \begin{pmatrix}
3 & 2 & 0 & -2 \\
2 & 1 & 0 & -2 \\
0 & 0 & 1 & 0 \\
2 & 2 & 0 & -1
\end{pmatrix},
\ 
\rho(z)= \begin{pmatrix}
3 & -2 & -2 & 0 \\
2 & -1 & -2 & 0 \\
-2 & 2 & 1 & 0 \\
0 & 0 & 0 & 1
\end{pmatrix} \tag{$\star$}
\end{equation}

\subsection{Proof of Theorem \ref{thm:H1H4}}
When considering infinitesimal deformations of the complete hyperbolic structure of the Borromean rings complement into $SO^+(4,1)$, we have the following result. 
\begin{theorem}[Theorem \ref{thm:H1H4}]
    For $\Gamma = \pi_{1}(S^{3} \backslash 6^{3}_{2})$, $$dim_{\mathbb{R}} H^{1}(\Gamma, \mathbb{R}^{3,1}) = 3\  \text{and}\ dim_{\mathbb{R}} PH^{1}(\Gamma, \mathbb{R}^{3,1}) = 0.$$
    This space of infinitesimal deformations is spanned by bending deformations supported along six totally geodesic thrice-punctured spheres. 
\end{theorem}
\begin{proof}

This proof has two parts: first, a dimension count for the space of infinitesimal deformations; and second, a geometric argument that shows these deformations are supported on a subcomplex of $\pantsc$. For the first part, in order to compute $dim_{\mathbb{R}} PH^{1}(\Gamma, \mathbb{R}^{3,1})$, we make an explicit computation using Fox Calculus.  We include $SO^+(3,1) \hookrightarrow SO^+(4,1)$ reducibly to compute the deformations of $\rho$. The Fox Calculus computation establishes that $H^1(\Gamma, \mathbb{R}^{3,1})$ is $3$-dimensional. To see the code, please refer to Appendix \ref{app:fox}.

The second part makes use of Theorem \ref{thm:bbson1}. As $dim_{\mathbb{R}} H^{1}(\Gamma, \mathbb{R}^{3,1}) = 3$, it suffices to show that there are at least $3$ dimensions of deformations supported on the branched totally geodesic complex $\pantsc$. We do this here. 

\begin{figure}
    \centering
    \includegraphics[width=0.2\linewidth]{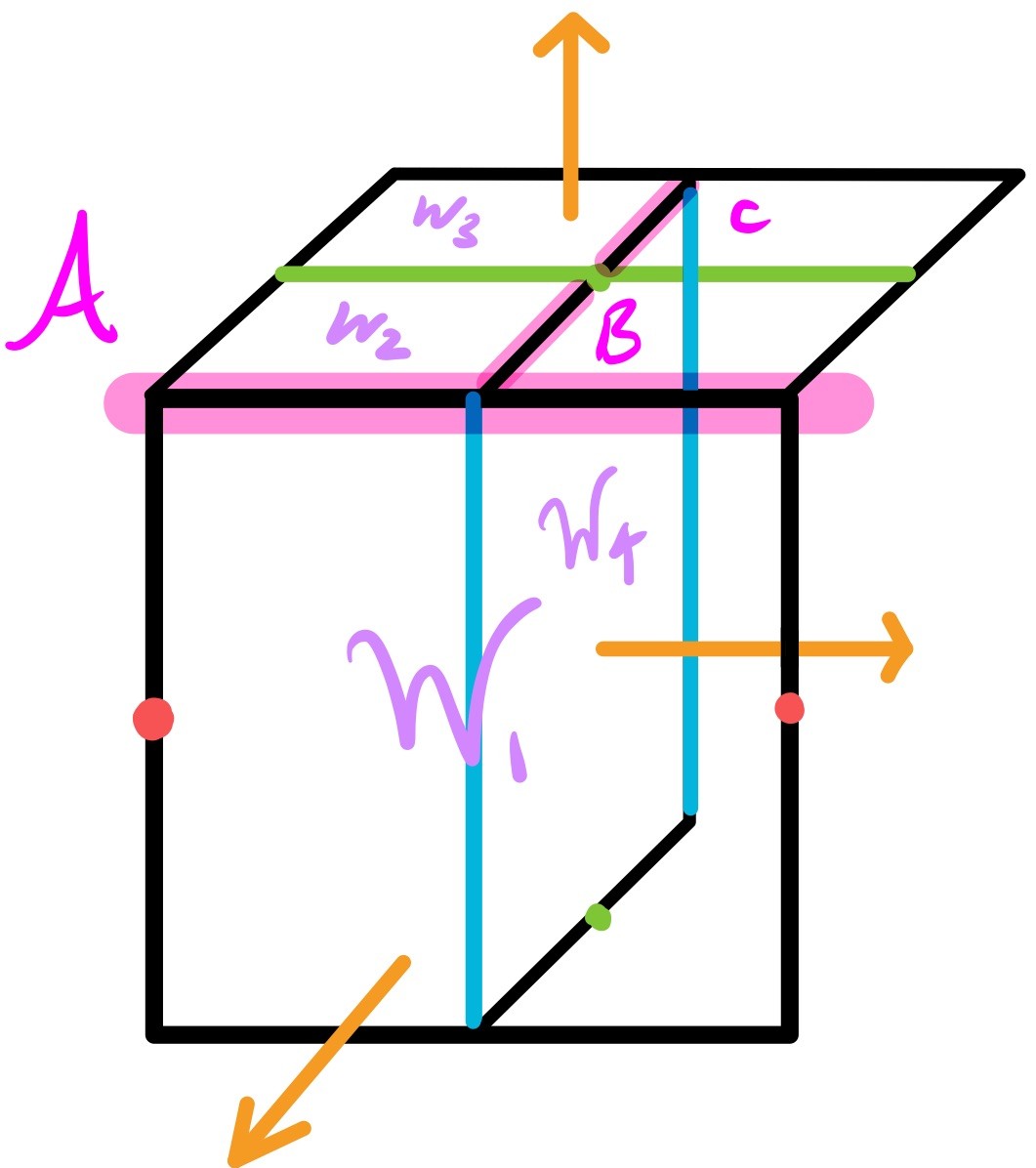}
    \caption{Four walls and three bindings making up a subcomplex of $\pantsc$; in particular, three surfaces of the six}
    \label{fig:so41comp}
\end{figure}

As depicted in Figure \ref{fig:so41comp}, we restrict to a subcomplex of $\pantsc$ for our computation. We have 4 walls, $w_1, w_2, w_3, w_4$ and 3 bindings $A, B, C$. This is because taking all possible walls in the complex $\pantsc$ means the cells $\delta \in \mathscr{C}$ will not have the large Fix$_{\Gamma}(\delta)$ we require in the hypotheses of Theorem \ref{thm:bbson1}; the cells of $M \backslash \pantsc$ are simply connected. However, by restricting to just the three surfaces depicted, we are able to guarantee this condition---each cell is bounded by at least two walls with Zariski-dense Fix$_{\Gamma}(w)$ in $SO^+(n-1,1)$. (For further discussion, see the remark following Lemma \ref{lem:faith}.)

An application of of Theorem \ref{thm:bbson1} is able to show: 
\[ dim_{\mathbb{R}} H^{1}(\Gamma, \mathbb{R}^{3,1}) \geq (4)-2(3) = -2\]
However, with a more detailed analysis of Equations \ref{eq:eq1} and \ref{eq:eq2} described in the proof of Theorem \ref{thm:bbson1} we are able to make a better bound; in particular, by also using the data of the angles between the walls. There are six computations to make: an evaluation of both Equation \ref{eq:eq1} and Equation \ref{eq:eq2} for each of the three bindings, and four variables, one for each wall.

For binding $A$, these are: 
\begin{align*}
    Eq_1(A)&: \ \omega_1 + \omega_3 \cos(\pi/2) + \omega_1 \cos(\pi) + \omega_2 \cos (3\pi/2) = 0 \\
    &\implies  0 = 0 \\
    Eq_2(A)&: \ \omega_3 \sin(\pi/2) + \omega_1 \sin(\pi) + \omega_2 \sin (3\pi/2) = 0 \\
    &\implies  \boxed{\omega_3 = \omega_2}
\end{align*}

For binding $B$, these are: 
\begin{align*}
    Eq_1(B)&: \ \omega_4 + \omega_2 \cos(\pi/2) + \omega_4 \cos(\pi) + \omega_2 \cos (3\pi/2) = 0 \\
    &\implies  0 = 0 \\
    Eq_2(B)&: \ \omega_2 \sin(\pi/2) + \omega_4 \sin(\pi) + \omega_2 \sin (3\pi/2) = 0 \\
    &\implies  \omega_2 = \omega_2
\end{align*}

For binding $C$, these are: 
\begin{align*}
    Eq_1(C)&: \ \omega_4 + \omega_3 \cos(\pi/2) + \omega_4 \cos(\pi) + \omega_3 \cos (3\pi/2) = 0 \\
    &\implies  0 = 0 \\
    Eq_2(C)&: \ \omega_3 \sin(\pi/2) + \omega_4 \sin(\pi) + \omega_3 \sin (3\pi/2) = 0 \\
    &\implies  \omega_3 = \omega_3
\end{align*}

From these calculations, we see that among our four variables, there is only one nontrivial equation, boxed above, and thus, we can now establish:
\[ dim_{\mathbb{R}} H^{1}(\Gamma, \mathbb{R}^{3,1}) \geq 3\]
But as we know $dim_{\mathbb{R}} H^{1}(\Gamma, \mathbb{R}^{3,1}) = 3$ from our Fox Calculus computations, we see $H^{1}(\Gamma, \mathbb{R}^{3,1})$ is realized by deformations supported on the subcomplex of $\pantsc$ shown, as desired. That $dim_{\mathbb{R}} PH^{1}(\Gamma, \mathbb{R}^{3,1}) = 0$ is also a Fox Calculus computation.

\end{proof}

\begin{remark*}
    Note that with just the data of the walls and branched loci, an application of Theorem \ref{thm:bbson1} is only able to show: 
\[ dim_{\mathbb{R}} H^{1}(\Gamma, \mathbb{R}^{3,1}) \geq (4)-2(3) = -2\]
While this is not a useful bound in this particular context, we remark that this bound is not achievable using Theorem 3.1 of \citet{Bart2006-im}; the complex we are considering does not satisfy their hypotheses as the walls here do \emph{not} have Zariski-dense stabilizer. 
\end{remark*}

\subsection{Proof of Theorem \ref{thm:HPG}} 
\label{subsec:sl4}

\begin{prop}\label{thm:H1PG}
    $dim_{\mathbb{R}} H^{1}(\Gamma, \nu_{4}) = 6$. Furthermore, this 6-dimensional space of infinitesimal deformations has a basis realized by bending deformations supported along the six totally geodesic thrice-punctured spheres described earlier. 
\end{prop}

\begin{proof}

    We include $SO^+(3,1) \hookrightarrow SL_{4}(\mathbb{R})$ to compute the deformations of $\rho$. 
    In order to compute $dim_{\mathbb{R}} H^{1}(\Gamma, \nu_{4})$ and $dim_{\mathbb{R}} PH^{1}(\Gamma, \nu_{4})$, we will again make an explicit computation using Fox Calculus. To see the computations and code, please refer to Appendix \ref{app:fox}. 

   Recall that for any choice of cusp $C_{i}$ in $\borro$, each has a torus cross-section $T_{i}$ such that $C_{i} = T_{i} \times [1, \infty)$, and each $T_{i}$ is Euclidean 2-torus. There are exactly four of our totally geodesic thrice-punctured spheres that intersect our chosen cusp, and each of these four surfaces either intersect this cusp torus in a pair of longitudes or a single meridian.  

   To show that the bending deformations supported on the six thrice-punctured spheres depicted in Figure \ref{fig:pantcube} are linearly independent, it suffices to construct a $6\times6$ matrix of full rank whose entries represent the effects of these deformations on six different curves in the manifold $\borro$.
   
With this data, we compute the bending deformation along each of these totally geodesic non-separating surfaces explicitly. In particular, by using the representation $\rho$ defined in (\ref{eq:so31}), we can compute a 1-parameter centralizer in $SL_4(\mathbb{R})$ for the fundamental group of each of these surfaces. In doing so, we can track how each of these six deformations affects six chosen curves in our manifold. Again, we denote the six totally geodesic thrice-punctured spheres, their fundamental groups, and associated stable letters when used as the separating surface in HNN-extension in Figure \ref{fig:pantstable}.

To be precise, we consider the following words: $x^{-1}y$, $xz$, $yz$, $xyz$, $xzy$, and $yzx^{-1}$. All of these words are loxodromic. We note that as seen in Remark 4.9 of \citet{BaDaLee}, it is not enough to verify that the eigenvalues of these words are changing to first-order. However, showing that the derivative of the traces are nonzero and linearly independent is sufficient. We construct a matrix whose rows correspond to each of the previously mentioned six words in $\pi_1(\borro)$ and whose columns correspond to a deformation supported along one of the six totally geodesic non-separating surfaces. In the $(i,j)$-entry of our matrix, we evaluate the derivative of the trace at the identity for the $i^{th}$ word being bent along the $j^{th}$ surface. We call this matrix $\mathcal{F}$. 

\[ \mathcal{F} = \left(
\begin{array}{cccccc}
 \frac{4}{3} & -\frac{4}{3} & -\frac{4}{3} & -\frac{28}{3} & 0 & 0 \\
 -\frac{28}{3} & -\frac{4}{3} & 0 & 0 & -\frac{4}{3} & \frac{4}{3} \\
 0 & 0 & \frac{4}{3} & -\frac{4}{3} & -\frac{28}{3} & -\frac{4}{3} \\
 -\frac{100}{3} & 20 & 20 & -\frac{100}{3} & -\frac{100}{3} & 20 \\
 -12 & -\frac{4}{3} & -\frac{4}{3} & -12 & -12 & -\frac{4}{3} \\
 \frac{164}{3} & \frac{4}{3} & -\frac{4}{3} & -12 & -12 & -\frac{4}{3} \\
\end{array}
\right) \]

The determinant of $\mathcal{F}$ is non-zero; thus the six bending deformations outlined above are linearly independent as infinitesimal deformations. As the Fox Calculus computation shows this is a $6$-dimensional deformation space, we can see the deformations represented by $\mathcal{F}$ form a basis as claimed. 

\end{proof}

\begin{figure}
\centering
\begin{tabular}{|c|c|c|}
\hline
\textbf{} $\mathcal{P}_{m,n}$ & $\pi_{1}$ & $*_{\alpha}$  \\ \hline
\textbf{} $(R,G)$& $\langle y^{-1}zy, z^{-1} \rangle$ & x  \\ \hline
          $(R,B)$& $\langle y^{-1}, zyz^{-1} \rangle$ & x  \\ \hline
          $(B,R)$& $\langle z^{-1}xz, x^{-1} \rangle$ & y   \\ \hline
          $(B,G$)& $\langle z^{-1}, xzx^{-1} \rangle$ & y   \\ \hline
          $(G,R)$& $\langle x^{-1}, yxy^{-1} \rangle$ & z   \\ \hline
          $(G,B)$& $\langle x^{-1}yx, y^{-1} \rangle$ & z   \\ \hline
\end{tabular}
    \caption{A totally geodesic non-separating surface in $\borro$, its $\pi_1$, and the stable letter of the associated HNN-extension}
    \label{fig:pantstable}
\end{figure}

\subsubsection{A Useful Lemma}

Here, we outline the effects of bending a hyperbolic cusp along an intersecting non-compact surface. This is based on Section 6 of \cite{Bobb2019-ct}; we refer the reader there for more detail. 

We say a submanifold $\Sigma$ in a hyperbolic manifold $M$ with a cusp $C \cong T \times [1, \infty)$ \emph{essentially intersects} $C$ if for any given cusp subneighborhood $C_t = T \times [t, \infty)$, ($C_t \cap \Sigma \neq \emptyset$). We are interested in computing the effects of bending along a pair of such surfaces $\Sigma_1$, $\Sigma_2$ such that they each essentially intersect a chosen cusp $C$, and intersect $T$ a torus cross-section of $C$ in parallel curves.

To be more precise, we say $(\Sigma_i \cap T) = \alpha_i$, a curve in $T$. We have arranged that $\alpha_1$ is parallel to $\alpha_2$ in $T$. In the example that matters for our computations, $\alpha_i$ will be a meridian (and for ease, we will choose the standard meridian). We want to compute the effects of bending along $\alpha_i$. In particular, we show that the effects of bending along a pair of such surfaces can be arranged to ``cancel'' when restricted to the cusp. 

\begin{lem} \label{lem:use}
Let $M$, $\Sigma_i$, $\alpha_i$, $C$, and $T$ be as above, and let a bending deformation into $SL_4(\mathbb{R})$ supported on $\Sigma_i$ be denoted by $\rho^i_t$. 
Then the composition of deformations $\rho^1_t$ and $\rho^2_{-t}$ restricted to $\pi_1(T)$ is induced by a conjugation. In particular, the infinitesimal deformation tangent to this composition of deformations is an element of $PH^1(\Gamma, \nu_4)$. 
\end{lem}

\begin{proof}
    
Since $\alpha_i$ is not separating in $T$, the bending deformation restricted to the cusp will be of HNN-extension type. Let $\pi_1(T)$ be generated by $\mu, \lambda$, corresponding to a meridian and longitudinal generator of $T$ respectively. 

Then, it remains to compute $\rho^i_t(\lambda)$, where $\rho^i_0$ is the initial representation $\rho: \pi_1(M) \hookrightarrow SO^+(3,1)$, $\rho^i_t$ is the $1$-parameter family of deformations obtained by bending along $\Sigma_i$, and $\lambda$ acts as the stable letter for the deformation. 
Let $c^i_t$ be the centralizer of $\alpha_i$ in $SL_4(\mathbb{R})$; then a bending deformation $\rho^i_t$ along $\alpha_i$ deforms $\lambda$ as:
\[ \rho^i_t(\lambda) = c^i_t\cdot\rho_0(\lambda)\]

One important remark: since $\alpha_1$ is parallel to $\alpha_2$, then if $c^1_t$ is the centralizer for $\alpha_1 \in SL_4(\mathbb{R})$, we can conjugate $c^1_t$ to get the centralizer of $\alpha_2$, $c_t^2$. In particular, we can choose a basis where the centralizer of $\alpha_2$ is $\tau c^1_t\tau^{-1}$, where
\[ 
\tau = \begin{pmatrix}
    1 & k & 0 & k^2/2 \\ 
    0 & 1 & 0 & k \\
    0 & 0 & 1 & 0 \\
    0 & 0 & 0 & 1 
\end{pmatrix}\]
and $k$ is a function of the distance between $\alpha_1$ and $\alpha_2$. 

Consider, in some chosen basis, 
\[
\rho_0(\lambda) = 
\begin{pmatrix}
     1 & x& 0 & x^2/2 \\ 
    0 & 1 & 0 & x\\
    0 & 0 & 1 & 0 \\
    0 & 0 & 0 & 1 
\end{pmatrix}
\]
Recall, our goal is to compute the effect of bending along $\Sigma_1$ and $\Sigma_2$ with the same magnitude but opposite orientation. Thus, we compute
\[ 
(\rho^1_t\circ\rho^2_t)(\lambda)  = c_t \cdot 
\tau c_{-t}\tau^{-1} = 
\begin{pmatrix}
    1 & x+(e^{-t}-1)k& 0 & x^2/2 + (e^{-t}-1)kx + (k^2 - e^{-t}k ) \\ 
    0 & 1 & 0 & x+(e^{t}-1)k\\
    0 & 0 & 1 & 0 \\
    0 & 0 & 0 & 1 
\end{pmatrix}
\]
which is in fact conjugate to $\rho_0(\lambda)$. 

Thus, since $\rho_t$ either acts as the identity or by conjugation for all generators of $\pi_1(T)$, and so by Theorem \ref{propA} we have that the infinitesimal deformation tangent to $\rho_t$ is an element of $PH^1(\Gamma, \nu_4)$ as desired. 

\end{proof}

In the proof of Theorem \ref{thm:H1PG}, we work with the explicit 1-parameter centralizer associated to each deformation. To see these centralizers and how they are computed, we refer the interested reader again to Appendix \ref{app:fox}. As a result of these explicit centralizers, we are able to compute the matrix $\mathcal{F}$; we also see that these bending deformations are in fact integrable. Moreover, as the surfaces in our system all either meet at right-angles or not at all, their deformations commute, as discussed in \cite{Bobb2019-ct}.

\begin{prop}\label{thm:PH1PG}
    $dim_{\mathbb{R}} PH^{1}(\Gamma, \nu_{4}) = 3$. Furthermore, this space is spanned by bending deformations supported along the six totally geodesic pants found in the complement of $\borro$. 
\end{prop}

\begin{proof}
   We continue with use of the notation defined in Theorem \ref{thm:H1PG}.

   There are three pairs of thrice-punctured spheres. In the notation from the table in Figure \ref{fig:pantstable}, they are $\{P_{(R,G)}, P_{(R,B)}\}$, $\{P_{(B,R)}, P_{(B,G)}\}$, and $\{P_{(G,R)}, P_{(G,B)}\}$. The first letter denotes the color of torus cusp component the surface intersects at a meridian. The second denotes which color the torus cusp component the surface intersects in a pair of (parallel, oppositely oriented) longitudes.

   We denote a deformation arising as a bending supported along one of these hypersurfaces, $\beta(P_{i,j})$. We construct three deformations, all independent of each other, that reside in the kernel of the restriction map to the torus boundary components, by taking sums of these pairs. 

   In particular, let $\beta(R) = \beta(P_{(R,G)}) -  \beta(P_{(R,B)})$, $\beta(B) = \beta(P_{(B,R)}) - \beta((P_{(B,G)})$, and $\beta(G) = \beta(P_{(G,R)}) - \beta(P_{(G,B)})$. These three deformations span  $PH^{1}(\Gamma, \nu_{4})$. 
   
   To see they are in $PH^1$, we again appeal to the notion of canceling pairs as in Lemma \ref{lem:use}. In the case of where a pair of pants meets the cusp in a pair of parallel longitudes, the parallel curves are oppositely oriented. Thus, their bending deformations cancel and the cusp experiences no bending by Lemma \ref{lem:use}.

    The deformations of these tori that arise from bending along the meridians %as in the construction in Theorem \ref{thm:H1PG} now 
    now cancel in pairs as well, by a similar argument to that in Lemma \ref{lem:use}. 

   That these deformations have nontrivial pre-image in $H^{1}(\Gamma, \nu_{4})$ and are independent follows from Theorem \ref{thm:H1PG}. This shows that they lie in $\ker res$. To see that they are the whole kernel, we recall that this space is $3$-dimensional, so their independence shows the desired result. 
\end{proof}
Propositions \ref{thm:H1PG} and \ref{thm:PH1PG} combined prove Theorem \ref{thm:HPG}.

\subsubsection{From Code to Cuspidal Cohomology}
We remark that the code documentation in Appendix \ref{app:fox} demonstrates a computational method of calculating the dimensions of the group cohomology we are interested in. We can evaluate our results against the following proposition:
\begin{prop}[\cite{Scannell2002-ve}, Prop 2.3] \label{prop:sum}
Let $M$ be a compact, oriented 3-manifold with fundamental group $\Gamma$, and let $V$ be a vector space of characteristic zero acted upon by $\Gamma$. Assume $\partial M$ consists of tori. Then 
\[ dim \ H^1(\Gamma, V) = dim \ker \text{res} \ + dim\ H^0(\pi_1(\partial M),V) \]
where \text{res} denotes the restriction on first cohomology. 
\end{prop}
In particular, by rearranging, we see that:
\[  dim \ H^1(\Gamma, V) - dim\  PH^1(\Gamma, V) = dim\ H^0(\pi_1(\partial M),V)  \]
When $V = \mathbb{R}^{3,1}$, then $H^0(\pi_1(\partial M),V)$ is known to have 1-dimensional contribution for each cusp \cite{Bart2006-im}, and so is in total $3$-dimensional in our setting. When $V = \nu_4$, then $H^0(\pi_1(\partial M),V)$ is also known to have 1-dimensional contribution for each cusp \cite{Garcia2025-vp}, and so we also obtain a total of $3$ dimensions in this setting. In Theorem \ref{thm:H1H4}, we see the difference in dimension between $H^1$ and $PH^1$ is indeed three; the same holds for Theorem \ref{thm:HPG}.

\subsection{Strictly Convex Projective Structures}
We recall a few facts about convex projective structures here. A \emph{convex projective domain} is an open subset $\Omega \subset \mathbb{RP}^{n}$ which is properly contained in and convex in an affine chart. We say such a set $\Omega$ is \emph{properly convex} if its closure $\bar\Omega$ is also contained in an affine patch. We say a set $\Omega$ is \emph{strictly convex} if $\partial\Omega$ does not contain a nontrivial projective segment. 

We denote the subgroup of $PGL_{n+1}({\mathbb{R}})$ that preserves a properly convex set $\Omega$, as $PGL_{n+1}({\Omega})$. Let $\Gamma$ be a discrete and torsion free subgroup of $PGL_{n+1}({\Omega})$. Then $\Omega\backslash\Gamma$ is a manifold that inherits a \emph{convex projective structure} or is a \emph{convex projective manifold.} If $\Omega$ were strictly convex, we say the result is a \emph{strictly convex projective manifold}. 

As discussed in Section \ref{prelim}, $\H{n}$ is a subgeometry of $\mathbb{RP}^{n}$; in particular, it is the Hillbert metric on an open ellipsoid in an affine patch. Thus, all complete hyperbolic structures on a manifold are also strictly convex projective structures. Moreover, when considering closed hyperbolic manifolds, small deformations of the complete hyperbolic structure will result in strictly convex projective structures \cite{Koszul1968-gt}, \cite{benoist2004convexes}. (Note that these deformations are not guaranteed to exist, as shown in \cite{Cooper2007-cs}.)

However, in the case of cusped hyperbolic manifolds, things are a bit different. It is known due to work of \citet{Marquis2012-kj} that bending in this setting does give rise to properly convex projective structures; however, in \cite{Ballas2020-nm}, there examples of bending deformations of cusped hyperbolic manifolds that do \emph{not} give rise to strictly convex projective structures. The question then becomes when bending gives rise to such structures. 

In the same work, \citet{Ballas2020-nm} show that there are conditions on the pair $(M, \Sigma)$, where $M$ is a cusped hyperbolic $n$-manifold and $\Sigma$ is an embedded totally geodesic hypersurface in $M$ for which we can do (projective) bending. These conditions determine if the deformations give rise to strictly convex projective structures. In particular, the authors introduce the notions of a \emph{standard cusp} and a \emph{bent cusp} and show that these are the only cusp types possible after bending. This fits into a larger result which classifies the generalized cusps possible for any cusped convex projective manifold \cite{Ballas2020-do}. Bent cusps serve as an obstruction to the deformation giving rise to a strictly convex projective structure. 

\begin{figure}
    \centering
\includegraphics[width=0.5\linewidth]{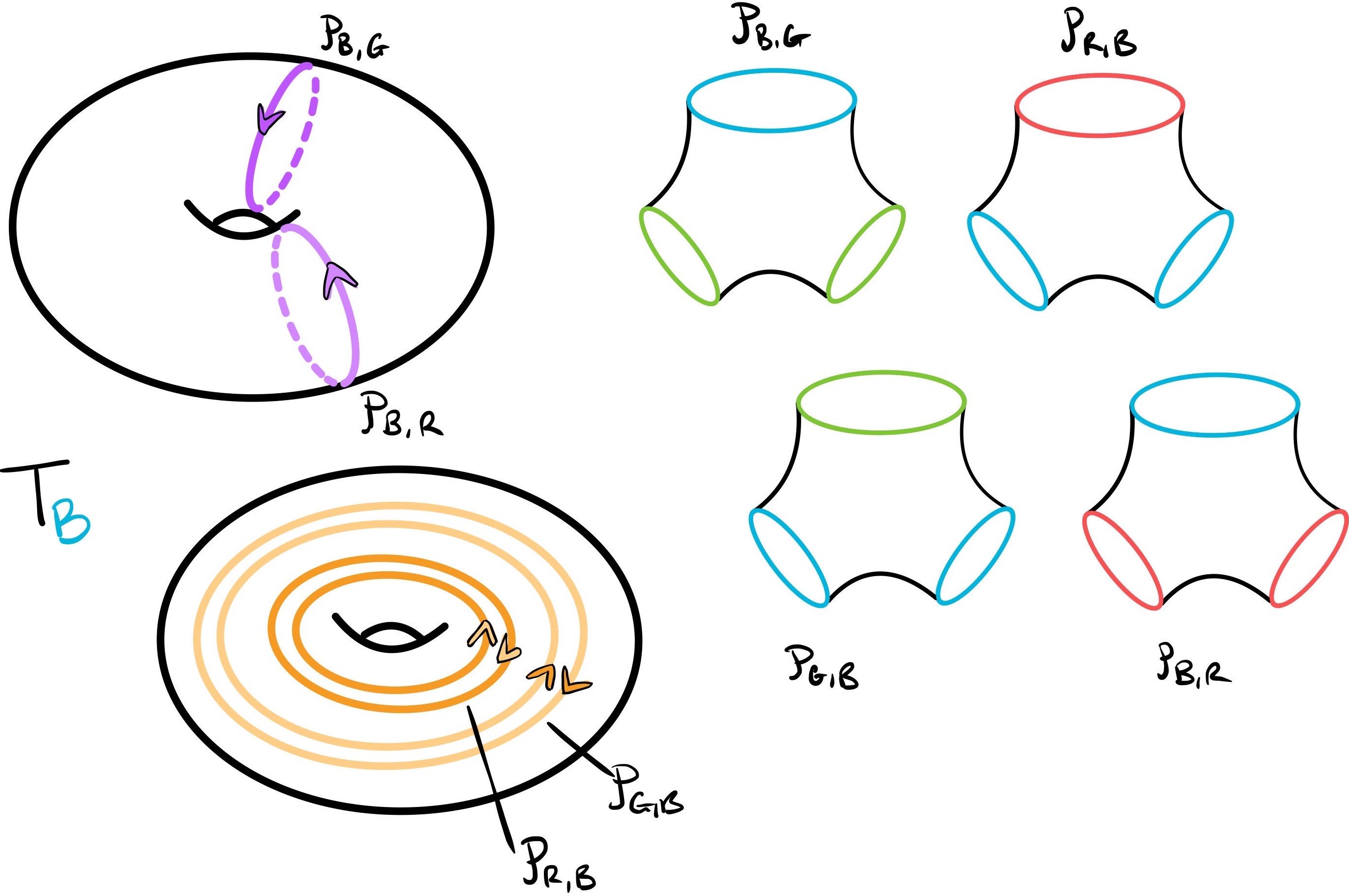}
    \caption{The totally geodesic thrice-punctured spheres that intersect the ``blue'' cusp and their intersections}
    \label{fig:cancel}
\end{figure}

The following results relate the topology of how $\Sigma$ sits in $M$ to the outcomes of bending, in the specific setting of hyperbolic $3$-manifolds. In this result, the notation $\mathfrak{X}_{scp}$ ($\Gamma$, $\text{PGL}_{n+1}$ ($\mathbb{R}$)) refers to the strictly convex projective representations of $\Gamma$ in $PGL_{n+1}(\mathbb{R})$. Also, the curve $\alpha$ is one transverse to $\Sigma \cap\partial M$; we are interested in computing the signed intersection $\iota$($\alpha$, $\Sigma \cap T_{i}$). 

\begin{prop}[Prop 3.2.7 in \cite{Ballas_2013}] \label{prop:sam}
Let $\Sigma$ be a non-compact, finite-volume, totally geodesic hypersurface of a cusped, finite-volume hyperbolic $3$-manifold $M$.  Suppose that each cusp $C_i$ of $M$ is diffeomorphic to $T_i$ $\times$ $[1,\infty)=0$, where $T_{i}$ is a torus. If $\rho_{t}$ is the family of representations obtained by bending along $\Sigma$ then $[\rho_{t}]$ is contained in $\mathfrak{X}_{scp}$  ($\Gamma$, $\text{PGL}_{n+1}$ ($\mathbb{R}$)) if and only if $\iota$($\alpha$, $\Sigma \cap T_{i}$) $=0$ for each $i$. 
\end{prop}

As a corollary of Prop \ref{prop:sam} and Theorem \ref{thm:PH1PG}, we have the following: 
\begin{corollary}[Cor. \ref{cor:proj}] 
Recall that $dim_{\mathbb{R}} PH^{1}(\Gamma, \nu_{4}) = 3$. These deformations are all integrable; furthermore, they are all strictly convex projective deformations. 
\end{corollary}

\begin{proof}
    From Theorem \ref{thm:PH1PG}, we have that $dim_{\mathbb{R}} PH^{1}(\Gamma, \nu_{4}) = 3$; moreover, we establish that there is a basis for this deformation space realized by the bending deformations supported on ``cancelling pairs'' of surfaces. This basis was denoted \{$\beta(R),\  \beta(B),\ \beta(G)$\}. In the proof of Lemma \ref{lem:use}, we show that deformations taking the form that these basis elements do remain integrable, as they exist for the same amount of time $t$ as their component integrable deformations. The deformations in $PH^{1}(\Gamma, \nu_{4})$ are then in fact integrable, even in combination with each other, as the totally geodesic hypersurfaces we are bending along meet orthogonally. (See \cite{Bobb2019-ct} for a further discussion of this construction).

    The union of these cancelling pairs of surfaces, one for each cusp, is a non-compact, finite-volume, totally geodesic hypersurface of $M$, which we denote $\mathcal{P}_i$. For each cusp $C_i$, the intersection of the torus cross section $T_i$ and $\mathcal{P}_i$ is a pair of oppositely oriented meridians, as depicted in Figure \ref{fig:cancel}. Then our curve $\alpha$ can be a longitude, parallel to those also depicted in Figure \ref{fig:cancel} for our convenience. We choose this $\mathcal{P}_i$ for the $\Sigma$ in the Proposition \ref{prop:sam}, and now show that $\iota$($\alpha$, $\mathcal{P}_i \cap T_{j}$) $=0$ for a torus $T_j$ in each cusp $C_j$.

    First, we have that $\iota$($\alpha$, $\mathcal{P}_i \cap T_{i}$) $=0$ as desired, since $\alpha$ meets our two oriented curves each once with opposite sign. Then, it remains to inspect $\iota$($\alpha$, $\mathcal{P}_i \cap T_{j}$) when $i \neq j$. The intersection $\mathcal{P}_i \cap T_{j}$ is a pair of oppositely oriented longitudes in $T_j$ (such as seen in Figure \ref{fig:cancel}). Thus by a similar argument to previous, $\iota$($\alpha$, $\mathcal{P}_i \cap T_{j}$) $=0$. 

    Thus, we have shown that for each such $\mathcal{P}_i$, Proposition \ref{prop:sam} applies; there are three such surfaces, each shown to support a linearly independent bending deformation from the others, and because they are pairwise orthogonal, these deformations commute. 
\end{proof}

\section*{End Remarks}
\subsection*{Future Directions} Throughout, we have assumed that $n \geq 3$; it is work in progress to show that $n=2$ yields analogous results to Theorems \ref{thm:bbson1} and \ref{thm:bbpgln}. Ultimately, this would appear as infinitesimal bending deformations supported along piecewise totally geodesic graphs embedded in a surface. The author hopes to present these results in a context where their utility is immediately displayed, as was done here in the case of Section \ref{sec:defbor}.   

 Furthermore, there is still more to be learned from the techniques used in Section \ref{sec:defbor}. While the original Menasco-Reid conjecture was proposed in dimension three, there are higher dimensional analogues that are still open. For example, it is unknown if there exists a hyperbolic link complement of 2-tori in a closed (smooth) simply connected 4-manifold that contains a closed embedded totally geodesic hyperbolic 3-manifold \cite{Chu2023-xo}. 

 Finally, it is presently unknown to the author how many of the flexible 3-manifolds described in \cite{Cooper2007-cs} admit a branched bending complex that may be able to describe its deformation space. This would be an exciting direction to explore, especially in its connection to the original conjecture of Kourouniotis \cite{Mathematisches_Forschungsinstitut_Oberwolfach1985-bo}: that is, if the deformation space of a compact hyperbolic manifold cannot be described only by bending, when can it be described by branched bending?

\appendix

\section{A Note on Parabolic Cohomology}
% Be sure to set \appendix or \appendices in structural/body.tex
 \label{app:cusp}
This section serves as a disambiguation for the notions of \emph{cuspidal cohomology} and \emph{parabolic cohomology}  used throughout. There are a few definitions that have subtle distinctions in their implementations; and still a few more computational perspectives that give varying results. Here we give a discussion of and justify the choices made throughout.

Parabolic cohomology was first introduced \cite{Weil1964-nb}, and then further developed and utilized in \cite{Guruprasad1995GroupSG}, \cite{Kim1999TheSG}, and \cite{Lawton2008-vv}. We define \emph{parabolic cohomology} as follows. Let $\Gamma$ be a nonuniform lattice in $SO(3,1) \subset G$. Then $\mathbb{H}^{3}/\Gamma$ is a finite volume, cupsed hyperbolic $3$-manifold $M$. Then:
\[ PH^{1}(\Gamma , \mathfrak{g}_{Ad_{\rho}}) = PZ^{1}(\Gamma , \mathfrak{g}_{Ad_{\rho}})/B^{1}(\Gamma , \mathfrak{g}_{Ad_{\rho}}) \]
where 
\[PZ^{1}(\Gamma , \mathfrak{g}_{Ad_{\rho}}) = \{ c \in Z^{1}(\Gamma , \mathfrak{g}_{Ad_{\rho}})\  \Big\vert\  c\vert_{\gamma} \in B^{1}(\Gamma , \mathfrak{g}_{Ad_{\rho}}),\  \forall \gamma \ \text{parabolic} \} \]

On the other hand, we define the \emph{cuspidal cohomology} of $M$, denoted throughout here as $H^{1}_{\partial}(\Gamma , \mathfrak{g}_{Ad_{\rho}})$ as:
\[ \ker \text{res:}\  H^{1}(\Gamma , \mathfrak{g}_{Ad_{\rho}}) \to H^{1}(\partial \Gamma , \mathfrak{g}_{Ad_{\rho}})  \]
where $\mathfrak{g}$ is the Lie algebra of $G$. 

In the settings of \cite{Weil1964-nb}, \cite{Guruprasad1995GroupSG}, \cite{Kim1999TheSG}, and \cite{Lawton2008-vv}, the group $\Gamma$ is a surface group, and so as each cusp subgroup is generated by precisely one parabolic element, it is fairly straightforward to show that these notions are equivalent. However, for a general hyperbolic $n$-manifold, it is not obvious that these definitions coincide; there may be groups generated by parabolic elements for which asking that each generator is deformed trivially in $\mathfrak{g}$ is not the same as asking that the entire group is deformed trivially.

It is a theorem due to \citet{kapovich1994deformations} (see Theorem 3) that for $\mathfrak{so}(4,1)$, parabolic and cuspidal cohomology are indeed the same. It is to the best of the author's knowledge that this is the only result that explicitly addresses this identification in the dimension $\geq 3$ setting. 
As a mild strengthening of this result, we can also say:
\begin{prop} \label{propA}
For $\Gamma$ a nonuniform lattice in $SO(n,1) \subset SO(n+1,1)$, with $n \geq 3$
\[ PH^{1}(\Gamma , \mathfrak{so}(n+1,1)) = H^{1}_{\partial}(\Gamma , \mathfrak{so}(n+1,1)) \]
\end{prop}

Similarly, we can argue: 
\begin{prop} \label{propA2}
For $\Gamma$ a nonuniform lattice in $SO(n,1)  \subset SL_{n+1}(\mathbb{R})$, 
\[ PH^{1}(\Gamma , \slr{n+1}) = H^{1}_{\partial}(\Gamma , \slr{n+1}) \]
\end{prop}

Throughout this appendix, we use the quadratic form
\[Q_n = \begin{pmatrix}
     &&-1\\
    &I&\\
   -1 &&
\end{pmatrix}\] where $I$ is the $n-1$-dimensional identity matrix. This is a signature-$(n,1)$ quadratic form on $\mathbb{R}^{n,1}$.

\begin{proof}[Proof of \ref{propA}]
    First, by definition, we have that 
    \[  H^{1}_{\partial}(\Gamma , \mathfrak{so}(n+1,1)) \subseteq PH^{1}(\Gamma , \mathfrak{so}(n+1,1)) \] Thus, we seek to show the opposite containment. For this strategy, we follow the proof of \cite{kapovich1994deformations}; in particular, if $A$ is a copy of $\mathbb{Z}^{n-1} \subset SO(n,1)$, then it suffices to show: 
    \[ PH^{1}(A , \ \mathbb{R}^{n,1}) = 0 \]

    Let $u \in PH^{1}(A , \mathfrak{so}(n+1,1))$, and let $a_1, a_2, \dots, a_{n-1}$ be a generating set for $A$. Then 
\begin{align*}
   u(a_1)&=(1-a_1)\alpha_1 \\
   u(a_2)&=(1-a_2)\alpha_2 \\
  &\vdots\\
   u(a_{n-1})&=(1-a_{n-1})\alpha_{n-1} 
\end{align*}
    where each $\alpha_i \in \mathbb{R}^{n,1}$. We want to find $\alpha'$ such that  
    \begin{align} \label{eq:goala}
   u(a_1)&=(1-a_1)\alpha' \\
   u(a_2)&=(1-a_2)\alpha' \\
  &\vdots\\
   u(a_{n-1})&=(1-a_{n-1})\alpha' 
\end{align}
    We note that without loss of generality $u(a_{n-1})=0 $. 

    To find the desired $\alpha'$, we notice that, without loss of generality, it must be that 
%     \begin{align*}
%      \alpha'-  \alpha_1 &\in ker(1-a_{1}) \\ 
%      \alpha'-  \alpha_2 &\in ker(1-a_{2}) \\ 
%   &\vdots\\
%    \alpha' &\in ker(1-a_{n-1}) 
% \end{align*}
% or alternatively, 
\begin{align*}
     \alpha'  &\in \alpha_1 + ker(1-a_{1})  \\ 
     \alpha'   &\in \alpha_2 + ker(1-a_{2}) \\ 
  &\vdots\\
   \alpha' &\in ker(1-a_{n-1}) 
\end{align*}

Thus, it remains to find an element in the intersection of these $n-1$ subspaces of codimension-2  in $\mathbb{R}^{n,1}$. The main argument is to find the intersection of all $ker(1-a_{i})$, and then show that after adding $\alpha_i$-affine translations, the intersection is still non-empty. 

We can write a general form of a parabolic generator as
\[ a_i = \begin{pmatrix}
1 & v_i^T &\sigma\\
 & I & v_i\\
&& 1 
\end{pmatrix}	
\]
and so
\[ (I-a_i) = \begin{pmatrix}
\  & v_i^T &\sigma\\
 &  & v_i\\
&& 
\end{pmatrix}	
\]
where $v_i \in \mathbb{R}^{n-1}$. Note that as all of the $a_i$ commute with each other, they are also all simultaneously upper-triangularizable. Then there is a general form for $ker(1-a_i)$ as
\[ ker(1-a_i) = \text{Span}\left\{ \begin{pmatrix}
    m_i \\
    0
    \\0
\end{pmatrix}, 
\begin{pmatrix}
    0 \\
    n_i\\
    0
\end{pmatrix}
\right\} \Bigg|\  m\in\mathbb{R},\  \langle n, v_i \rangle = 0 \] We assume without loss of generality that the generating $a_i$ are independent; in particular, that the $v_i$ are linearly independent of each other. 

As all the $a_i$ generators commute, we have that $u(a_ia_j) = u(a_ja_i)$. From this, we obtain
\begin{align*}
    u(a_ia_j) &= u(a_ja_i)\\
     u(a_i)+a_iu(a_j) &= u(a_j)+a_ju(a_i)  \\
     %u(a_i)-a_ju(a_i) &= u(a_j)-a_iu(a_j) \\
     (1-a_i)u(a_j) &= (1-a_j)u(a_i) \\
     %(1-a_i)(1-a_j)\alpha_j &= (1-a_j)(1-a_i)\alpha_i \\
     (1-a_i)(1-a_j)\alpha_j &= (1-a_i)(1-a_j)\alpha_i
\end{align*}
for any two $\alpha_i, \ \alpha_j$. Further calculation shows
\[ (1-a_i)(1-a_j) = \begin{pmatrix}
    0 &0& v_i^Tv_j \\
    0&0&0\\
    0&0&0
\end{pmatrix}\]
which, when evaluated on $\alpha_i$ and $\alpha_j$, gives the condition $y_i=y_j$ for any two $\alpha_i, \ \alpha_j$, unless $\langle v_i, v_j \rangle =0.$ However, a change-of-basis on the initial generating set is able to guarantee no such collisions occur. Thus, without loss of generality, 
\[\alpha_i = 
\begin{pmatrix}
    x_i \\ 
    w_i \\
    0
\end{pmatrix}\ \text{such that }  x_i\in\mathbb{R},\ w_i \in \mathbb{R}^{n-1}  \]

With a description of each $u(a_i)$ and $ker(I-a_i)$, we now seek to establish the mutual $\alpha'$ described previously. We proceed by induction. For the initial case, consider 
\begin{align}
    u(a_1) &= (1-a_1)\alpha_1 \\
    u(a_2) &= (1-a_2)\alpha_2    
\end{align}
We need to find elements $k_1 \in ker(1-a_1)$, $k_2 \in ker(1-a_2)$ such that
\[ \alpha'= \alpha_1 +k_1 = \alpha_2 + k_2 \]
This is equivalent to showing that there exists some $\beta \in \mathrm{Span}\{ker(1-a_1),\ ker(1-a_2)\}$ such that 
\[  (1-a_1)(1-a_2)(\alpha_2+\beta) = (1-a_1)(1-a_2)\alpha_1\]
There exists $\beta \in ker(1-a_1)(1-a_2)$, which contains both $ker(1-a_1)$ and $ker(1-a_2)\}$---it then remains to show that 
\[ \mathrm{Span}\{ker(1-a_1),\ ker(1-a_2)\} = ker\Big((1-a_1)(1-a_2)\Big)\]
for the $a_1,\ a_2$. 

We note that $ker\Big((1-a_1)(1-a_2)\Big)$ is codimension-$1$ in $\mathbb{R}^{n,1}$---the kernel is all but $e_{n+1}$ in the chosen basis. We earlier showed that $ker(1-a_i)$ is codimension-$2$. The result follows if $ker(1-a_1) \neq ker(1-a_2)$, which is a consequence of the hypothesis that the $v_i$ are all linearly independent of each other.  

Now, we assume that for $a_1, \dots, a_k$, we have arranged
\[ \alpha = \alpha_1 +k_1 = \alpha_2 + k_2 = \cdots = \alpha_k + k_k\]
We want to show there exists $\alpha' \in \mathbb{R}^{n,1}$ , $k' \in \bigcap^k ker(1-a_i)$, and $k_{k+1} \in ker(1-a_{k+1})$ such that
\[ \alpha' = (\alpha_1 +k_1) + k' = (\alpha_2 + k_2) + k' = \cdots = (\alpha_k + k_k)  + k' = \alpha_{k+1} + k_{k+1} \] Similarly, this is equivalent to checking that 
\[ \mathrm{Span}\{\bigcap^k_j ker(1-a_j),\ ker(1-a_{k+1})\} = ker\Big((1-a_i)(1-a_{k+1})\Big)\]
for $i \in [1,k]$. But once again, there is a containment of the space on the left in the space on the right; since $ker(1-a_{k+1})$ is codimension-$1$ in $ker\Big((1-a_i)(1-a_{k+1})\Big)$, it suffices to show $\bigcap^k_j ker(1-a_j) \nsubseteq ker(1-a_{k+1})$. This again guaranteed by the independence of the $v_i$. The desired result then follows. 
\end{proof}

\begin{proof}[Proof of \ref{propA2}]
    The procedure for $\nu_{n+1}$ coefficients is very similar. For the sake of brevity, we use much of the same notation and conventions from the previous argument.
    
    We highlight the key differences here. First,  with respect to the quadratic form $Q_n$, the set of matrices in $\nu_{n+1}$ is the set of $V \in \mathfrak{sl}_{n+1}(\mathbb{R})$ given by: 
    \begin{align*}
        V^TQ_n - Q_nV &= 0 \\
         V^TQ_n &= Q_nV \\
         Q_nV^TQ_n &= V
    \end{align*}
    This space is $\frac{(n+1)(n+2)}{2}-1$-dimensional. (Or, in other words, $\frac{n^2+3n}{2}$-dimensional.)

    The space of the $\alpha$ with respect to this quadratic form can be described as follows: 
    \[\alpha_i = \left(
    \begin{NiceArray}{c|c|c}
     A_i \ & W_i^T  \ & B_i \ \\
     \hline  X_i  \   & \phantom{\Bigg|00} E_i \phantom{00}  &  -W_i    \    \\
   \hline  C_i \ &   -X^T_i \  & -A_i \
    \end{NiceArray}
 \right) \]
 where here, $\alpha_i$ is represented as an $(n+1) \times (n+1)$-matrix; $A_i,\  B_i,\ C_i$ are $1\times1$-matrices, $X_i$ and $W_i$ are $(n-1)\times1$-matrices, and $E_i$ is a traceless matrix with antisymmetry across the antidiagonal with $(n-1)\times(n-1)$-dimension. 
    
 The conditions that $u(a_ia_j)=u(a_ja_i)$ impose on the $\alpha_i$, without loss of generality, are the conditions that $X_i$ and $C_i$ are 0, and that $E_i$ is also symmetric.  That is
 \[\alpha_i = \left(
    \begin{NiceArray}{c|c|c}
     A_i \ & W_i^T  \ & B_i \ \\
     \hline  0  \   & \phantom{\Bigg|00} E_i \phantom{00}  &  -W_i    \    \\
   \hline  0 \ &   0 \  & -A_i \
    \end{NiceArray}
 \right) \]

    Similarly, for the description of $ker(1-a_i)$ we obtain
 \[
   ker(1-a_i) = 
   \textrm{Span}\left\{ 
   \left(
    \begin{NiceArray}{c|c|c}
     0 \ & W ^T  \ & B \ \\
     \hline  0  \   & \phantom{\Bigg|00} E \phantom{00}  &  -W    \    \\
   \hline  0 \ &   0 \  & 0 \
    \end{NiceArray}
 \right)  
   \Bigg| \ E(v_i) = 0  
   \right\} 
  \]
   which mirrors the conditions on $ker(1-a_i)$ found in the previous case. 

   The induction is nearly identical---the only step we need to verify is  that 
\[ \mathrm{Span}\{\bigcap^k_j ker(1-a_j),\ ker(1-a_{k+1})\} = ker\Big((1-a_i)(1-a_{k+1})\Big)\]
for $i \in [1,k]$. The $ker\Big((1-a_i)(1-a_{k+1})\Big)$ has description
\[ \textrm{Span}\left\{ 
   \left(
    \begin{NiceArray}{c|c|c}
     0 \ & W ^T  \ & B \ \\
     \hline  0  \   & \phantom{\Bigg|00} E \phantom{00}  &  -W    \    \\
   \hline  0 \ &   0 \  & 0 \
    \end{NiceArray}
 \right)  
   \Bigg| \ E(v_i) = 0,\  E(v_{k+1})=0  
   \right\}  
    \] and so as before, $\mathrm{Span}\{\bigcap^k_j ker(1-a_j),\ ker(1-a_{k+1})\} \subseteq ker\Big((1-a_i)(1-a_{k+1})\Big)$. We rely once more on the linear independence of the $v_i$ to establish the claim. 

\end{proof}

\section{Fox Calculus Computations}
 \label{app:fox}
This section serves as a directory for the Mathematica notebooks. 

To see the computations for Theorem \ref{thm:H1H4}: 
\url{https://websites.umich.edu/~cdmonroe/mathematica/r31.nb}. 

To see the computations for Theorem \ref{thm:HPG}: 
\url{https://websites.umich.edu/~cdmonroe/mathematica/h1_into_nu4.nb}. 

To see the computations for Section \ref{subsec:sl4}: 
\url{https://websites.umich.edu/~cdmonroe/mathematica/HNN_sl4r.nb}.

\bibliographystyle{abbrvnat}
\bibliography{references}  

\end{document}